\documentclass[leqno,12pt]{amsart}

\usepackage{amsfonts,amssymb,amsgen,stmaryrd, amsopn}
\usepackage[leqno]{amsmath} 
\usepackage[T1]{fontenc}
\usepackage[latin1]{inputenc} 
\usepackage[french,english]{babel}
\usepackage{graphicx}
\usepackage{graphpap}
\usepackage{hyperref}
\newtheorem{theoreme}{Theorem}[section]

\newtheorem{proposition}{Proposition}[section]
 \newtheorem{lemme}{Lemma}[section]
 
 \newtheorem{rem}{Remark}[section]
 
\renewcommand\Re{\mathrm{Re}\,} \renewcommand\Im{\mathrm{Im}\,}
\newcommand\R{{\mathbb R}} \newcommand\N{{\mathbb N}}
\newcommand\Z{{\mathbb Z}} 
\newcommand\C{{\mathbb C}} 
\newcommand\T{{\mathbb T}}

\renewcommand{\Im}{  \text{Im}   }
\renewcommand{\Re}{  \text{Re}   }

 \def\cdotv{\raise 2pt\hbox{,}}

\renewcommand\Re{\mathrm{Re}\,} \renewcommand\Im{\mathrm{Im}\,}

\makeatletter
\def\@tvsp{\mathchoice{{}\mkern-4.5mu}{{}\mkern-4.5mu}{{}\mkern-2.5mu}{}}
\def\ltrivert{\left|\@tvsp\left|\@tvsp\left|}
\def\rtrivert{\right|\@tvsp\right|\@tvsp\right|}
\makeatother

 \def\cdotv{\raise 2pt\hbox{,}}

\begin{document}
\title[
Transport  of gaussian measures by the NLS flow 
 ]
{
Transport  of gaussian measures by the flow of the nonlinear  Schr\"odinger equation
}
  \author{Fabrice Planchon}
  \address{    Universit\'e C\^ote d'Azur, CNRS, LJAD, France}
  \email{fabrice.planchon@univ-cotedazur.fr} 
\author{Nikolay Tzvetkov}
\address{
Universit\'e de Cergy-Pontoise,  Cergy-Pontoise, F-95000,UMR 8088 du CNRS
}
\email{nikolay.tzvetkov@u-cergy.fr}
\author{Nicola Visciglia}
\address{Dipartimento di Matematica, Universit\`a di Pisa, Italy}
\email{nicola.visciglia@unipi.it}
\thanks{ The first author was partially supported by ERC grant ANADEL no 757996, 
the second author by ANR grant ODA (ANR-18-CE40-0020-01), and the third 
author by PRA 2016 Problemi di Evoluzione: Studio Qualitativo e Comportamento Asintotico} 
\date{\today}
 \maketitle
 \par \noindent

\begin{abstract}
We prove a new smoothing type property for solutions of the 1d quintic Schr\"odinger equation. 
As a consequence, we prove that a family of natural gaussian measures are quasi-invariant under the flow of this equation. In the defocusing case, we prove global in time quasi-invariance while in the focusing case we only get local in time quasi-invariance because of a blow-up obstruction.  
Our results extend as well to generic odd power nonlinearities. 
\end{abstract}
\section{Introduction}
\subsection{}
Our goal here is to develop further techniques we introduced in our previous work \cite{PTV} in the context of the 1d quintic defocusing NLS.
This will allow to prove quasi-invariance for a  family of natural gaussian measures under the flow of the 1d quintic defocusing NLS on the one dimensional torus. 
This is a significant generalization of the recent works \cite{OST,OT1,OT2,OTT,Tz} since the NLS case was out of reach of the techniques used there. 
Moreover, in the focusing case we get local in time quasi-invariance, thus answering  a question first raised by Bourgain in \cite[page~28]{Bo_proceeding}.    
The quintic NLS exhibits the simplest power like nonlinearity giving a non integrable equation, but our results may easily be extended to all odd power nonlinearities. We elected to focus on the quintic case merely for the sake of clarity in the arguments.
\subsection{A statistical view point on the linear Schr\"odinger equation on the torus}
Consider 
\begin{equation}\label{1}
(i\partial_t+\partial_x^2)u=0,\,\, u(0,x)=u_0(x), \quad t\in \R,\,\, x\in\T. 
\end{equation}
The solution of \eqref{1} is given by the Fourier series 
\begin{equation}\label{fs}
u(t,x)=\sum_{n\in\Z} e^{inx}\, e^{-itn^2}\, \widehat{u}_0(n),
\end{equation}
where $\widehat{u}_0(n)$ denote the Fourier coefficients of $u_0$. It is well-known that such a Fourier series can have quite a complicated behaviour, in particular with respect to rational and irrational times (see e.g. \cite{KR}). It turns out that the situation is much simpler if one adopts a statistical view point on \eqref{fs}. 
More precisely, if we suppose that $ \widehat{u}_0(n)$ are distributed according to independent, centered complex gaussian variables then, using the invariance of complex gaussians under rotations, we observe that \eqref{fs} is identically distributed for each time $t$;  under the considered statistics the quantum particle does not know the time (here we use the terminology of \cite{KR}).
\\

Let us now consider a natural family of gaussian measures associated with \eqref{1}. For $s\in\R$, we denote by $\mu_s$ the gaussian measure induced by the map
\begin{equation}\label{omeg}
\omega\longmapsto \sum_{n\in\Z}\frac{g_n(\omega)}{(1+n^2)^{s/2}}\, e^{inx} \,.
\end{equation}
In \eqref{omeg}, $(g_n)_{n\in\Z}$ is a family of independent, standard complex gaussians (i.e. $g_n=h_n+il_n$, where $h_n,l_n$ are independent and belong to ${\mathcal N}(0,1)$). 
One can see $\mu_s$ as a probability measure on $H^\sigma(\T)$, $\sigma<s-1/2$ such that $\mu_s(H^{s-\frac{1}{2}}(\T))=0$.
According to the above discussion, one can prove the  following statement. 
\begin{theoreme}\label{prop1}
For every $s\in\R$, the measure $\mu_s$ is invariant under the flow of \eqref{1}. 
\end{theoreme}
Formally, one may see $\mu_s$ as  a normalized version of $\exp(-\|u\|_{H^s}^2)du$ and the statement of Theorem~\ref{prop1} is then a consequence of the divergence free character of the vector field generating the flow of \eqref{1} and conservation of $H^s$ norms by the flow of \eqref{1}.
\subsection{A statistical view point on the defocusing NLS on the torus}
Consider now a  (non integrable) nonlinear perturbation of \eqref{1} :
\begin{equation}\label{NLS}
(i\partial_t+\partial_x^2)u=|u|^4 u,\,\, u(0,x)=u_0(x), \quad t\in \R,\,\, x\in\T. 
\end{equation}
One can ask how much the result of Theorem~\ref{prop1} extends to the case of \eqref{NLS}. We observe that since \eqref{NLS} is a nonlinear equation, well-posedness is  a non trivial issue compared to \eqref{1} where we have an explicit formula for the solutions.  Thanks to \cite{Bo1}, we know that \eqref{NLS} is (locally) well-posed in $H^s(\T)$, $s>0$. As a byproduct of the analysis of \cite{Bo2} we can deduce the following statement.
\begin{theoreme}\label{th1}
The measure $\mu_1$ is quasi-invariant under the (well-defined)  flow of \eqref{NLS}.
\end{theoreme} 
Here by quasi-invariance we mean that the transport of $\mu_1$ under the flow of \eqref{NLS} is absolutely continuous with respect to $\mu_1$.
This quite  remarkable property relies on dispersion in \eqref{NLS} in an essential way.
On the contrary, following \cite{OST} one would expect that for $s>1/2$, the measure $\mu_s$ is not quasi-invariant under the flow of the dispersionless version of \eqref{NLS}:
\begin{equation*}
i\partial_t u=|u|^4 u,\,\, u(0,x)=u_0(x), \quad t\in \R,\,\, x\in\T. 
\end{equation*}
Therefore the quasi-invariance property displayed by Theorem~\ref{th1} is a delicate property measuring the balance between dispersion and nonlinearity. 
Theorem~\ref{th1} follows from invariance of the Gibbs measure $\exp(-\frac{1}{6}\|u\|_{L^6}^6) d\mu_1(u)$ under the flow of \eqref{NLS} (see \cite{Bo2}). This Gibbs measure invariance uses  both basic conservation laws for \eqref{NLS} in a fundamental way: the $L^2$ (mass) conservation and the energy conservation, where the energy functional is given by
$
\frac{1}{2}\int_{\T} |\partial_x u|^2+\frac{1}{6}\int_{\T}|u|^6\,.
$
Therefore, for lack of higher order conservation laws for \eqref{NLS},  extension of Theorem~\ref{th1} to the family of measures $\mu_s$ defined by \eqref{omeg} is a non trivial issue. 
 
Our main result is that Theorem~\ref{th1} holds true for $s=2k$, where $k\geq 1$ is an integer. 
\begin{theoreme}\label{th2}
The measure $\mu_{2k}$ is quasi-invariant under the flow of \eqref{NLS}
for every integer $k\geq 1$.
\end{theoreme} 
It is worth mentioning that in the completely integrable case a huge literature has been devoted to prove invariance, and hence quasi-invariance, 
of (weighted)  Gaussian measures associated with higher order conserved energies 
(see for instance \cite{Z} for KdV and cubic NLS, \cite{DTV}, \cite{TV1}, \cite{TV2}
for the Benjamin-Ono equation, 
\cite{GLV1}, \cite{GLV2}, \cite{ONRS}, \cite{NRSS} for DNLS). Of course the main point in Theorem \ref{th2}
is that \eqref{NLS} is not completely integrable.
\\

In \cite{Tz} the author introduced a new method (inspired by \cite{DTV,TV1,TV2}) to study quasi-invariance of gaussian measures along the flow associated with dispersive equations. This approach was further generalized to much more involved situations in \cite{OST,OT1,OT2,OTT}. In particular, in \cite{OT2} a multi-linear stochastic argument was introduced.  In the present paper we follow the general strategy from \cite{Tz}, up to two new and crucial ingredients. 
\\

First, in the study of the measure evolution, we invoke a more deterministic approach based on knowledge about individual trajectories; in particular, working on regions of phase space that are left invariant by the flow is irrelevant. This approach forces us to work with $L^2$ type spaces but has the advantage to leave an important freedom in the quantitative bounds (see the discussion after Theorem~\ref{determinist}  below).
\\

Second, the most important novelty in our proof of Theorem~\ref{th2} is a subtle improvement of the modified energy method of our previous work \cite{PTV}. In order to state the precise estimate implied by this improvement,  we introduce  some notations.
For $M\geq 0$ an integer, we denote by $\pi_M$, the Dirichlet projector defined by
$$
\pi_M\Big(
\sum_{n\in\Z}c_n e^{inx}
\Big)
=
\sum_{|n|\leq M}c_n e^{inx}\,.
$$
We also use the convention that for $M=\infty$, $\pi_{M}={\rm Id}$.
We next consider the following truncated version of \eqref{NLS}
\begin{equation}\label{NLS_N}
(i\partial_t+\partial_x^2)u=\pi_M\big(|\pi_M u|^4 \pi_M u\big),\,\, u(0,x)=u_0(x), \quad t\in \R,\,\, x\in\T. 
\end{equation}
We denote the flow of \eqref{NLS_N} by $\Phi_M(t)$. Then $\Phi(t):=\Phi_{\infty}(t)$ is  the flow associated with our true nonlinear equation \eqref{NLS}.
We can now state the basic modified energy estimate that we shall need in order to prove Theorem~\ref{th2}. 
We believe that this statement is relevant for its own sake.
\begin{theoreme}\label{determinist}
Let $k\geq 1$ be an integer. 
There is an 
integer $m_0>0$
and a positive constant $C$ such that the following holds true. There exists a functional $E_{2k}(u)$ such that
$$
E_{2k}(u)=\|u\|_{H^{2k}}^2 + R_{2k}(u),\quad E_{2k}(0)=0
$$ 
with
\begin{equation}\label{NEW}
 |R_{2k}(u)-R_{2k}(v)|\leq C\|u-v\|_{H^{2k-1}}
\big( 1+ \|u\|_{H^{2k-1}}^{m_0}+\|v\|_{H^{2k-1}}^{m_0}\big)
\end{equation}
and moreover for every  $M\in \N\cup \{\infty\}$,
\begin{equation}\label{modifene}
\frac d{dt} E_{2k}(\pi_M \Phi_M(t)(u_0))= F^{(M)}_{2k}(\pi_M \Phi_M(t)(u_0)),
\end{equation}
where the functional $F^{(M)}_{2k}(u)$ satisfies 
\begin{equation}\label{NEW_pak_2k}
 |F^{(M)}_{2k}(u)|\leq C \big( 1+ \|u\|_{H^{2k-1}}^{m_0}\big),\quad k\geq 2\,,
\end{equation}
and
\begin{equation}\label{NEW_pak_2}
 |F^{(M)}_{2}(u)|\leq C \big( 1+ \|u\|_{H^{1}}^{m_0}\big)\big(1+\|\partial_x u\|^4_{L^4}\big)\,.
\end{equation}
\end{theoreme}
We believe that \eqref{NEW_pak_2k} can be improved to a tame estimate but this is not of importance for our purposes. 
\\

Estimates \eqref{NEW_pak_2k} and \eqref{NEW_pak_2} show that the quintic NLS enjoys a new form of one derivative smoothing with respect to the nonlinearity.  Such a one derivative smoothing is of course well-known for the nonlinear wave equation NLW and  it may easily be seen when we write NLW as a first order system.   We find it remarkable that NLS  which a priori does not have the favorable structure of NLW  satisfies such a  (nonlinear) one derivative smoothing when writing suitable modified energy estimates that we will introduce later on.
\\

In the next section we prove that Theorem~\ref{determinist}  implies Theorem~\ref{th2} by a generic argument.
Let us observe that estimates \eqref{NEW_pak_2k} and \eqref{NEW_pak_2} imply the bound
\begin{equation}\label{NEW_pak_again}
 |F^{(M)}_{2k}(u)|\leq C \big( 1+ \|u\|_{H^{2k-\frac 12 -\epsilon}}^{m_0}\big),
\end{equation}
for a suitable $\epsilon=\epsilon(k)>0$ and $m_0=m_0(k)$. 
In the remaining part of this paper we will consider $\mu_{2k}$ as a measure on the space $H^{2k-\frac{1}{2}-\epsilon}$ with the choice of $\epsilon$ such that \eqref{NEW_pak_again} holds.
\\

It is worth comparing the approach in this paper to that of \cite{Tz}.  In order to apply the method of \cite{Tz}, 
it suffices to prove Theorem~\ref{determinist} with $F^{(M)}_{2k}(u)$ satisfying 
\begin{equation}\label{OLD}
|F^{(M)}_{2k}(u)|\leq C \big(\|u\|_{H^1}\big)\big( 1+ \|u\|_{C^{2k-\frac 12 -\epsilon}}^{2}\big)\,,
\end{equation}
where $C^{\alpha}$ denotes the H\"older space of order $\alpha$ and a similar estimate for $R_{2k}(u)$.
Observe that for $m_0\leq 2$ we have that  \eqref{NEW_pak_again} implies \eqref{OLD}.
For $m_0>2$ neither of \eqref{NEW_pak_again} and \eqref{OLD} implies the other. 
The estimate \eqref{OLD} has the advantage to involve a stronger $C^{2k-\frac 12 -\epsilon}$  
norm on the righthand side while estimate \eqref{NEW_pak_again} has the great advantage to allow any power $m_0$. 
In the context of the problem considered in this paper estimate \eqref{NEW_pak_again} 
was easier to achieve for us, especially for large values of $k$.
\\

Let us also mention that a difficulty that was  present in \cite{OT2} and which we do not face here is the need of {\it renormalisations} in the energy estimates. 
The main novelty in our present analysis when compared to \cite{OT2} (and also \cite{PTV}) is the key introduction of  several new correction terms in the {\em modified energies} that we shall construct, which allow to finally get the smoothing displayed by Theorem~\ref{determinist}. 
Let us also mention \cite{OST,OT1}  where the correction terms in the energy estimates are constructed via normal form transformations. This approach exploits the smoothing coming from the non resonant part of the nonlinearity. It would be interesting to find situations where the approaches of  \cite{OST,OT1} and the one used in this paper can collaborate.
We finally mention \cite{Kwon}, \cite{Kenig}, \cite{Koch}, \cite{OV} where modified energy techniques have been used in the context of dispersive equations.
\\

We believe that Theorem~\ref{th2} can be extended to odd integers  by further elaborations on our arguments. 
Finally, we have some hope to extend our results to $2d$. In this case renormalisation arguments in the spirit of \cite{OT2} would most certainly be required.
\subsection{Local in time quasi-invariance in the focusing case}
In this section, we consider the focusing equation
\begin{equation}\label{NLS_focus}
(i\partial_t+\partial_x^2)u=-|u|^4 u,\,\, u(0,x)=u_0(x), \quad t\in \R,\,\, x\in\T. 
\end{equation}
Because of blow-up phenomena  (see e.g. \cite{OgTs}), we cannot expect that the flow of \eqref{NLS_focus} is $\mu_{2k}$ a.s. defined (globally in time). 
We only have local well-posedness on the support of $\mu_{2k}$ and  quasi-invariance of 
our gaussian measures may only hold locally. Here is a precise statement.
\begin{theoreme}\label{th3}
Let $k\geq 1$ be an integer. 
For every $R>0$ there is $T>0$ such that the following holds true. For every $u_0\in H^{2k-\frac 12 - \epsilon}$ such that
$
\|u_0\|_{H^{2k-\frac 12 - \epsilon}}<R
$
there is a unique solution $u$ of \eqref{NLS_focus} in $C([-T,T];H^{2k-\frac 12 - \epsilon})$. 
Moreover, if 
$$
A\subset \Big \{u\in H^{2k-\frac 12 - \epsilon} \,|\,  \|u\|_{H^{2k-\frac 12 - \epsilon}}<R\Big\}
$$ 
is such that $\mu_{2k}(A)=0$
then  
$$
\mu_{2k}(A_t)=0,\quad t\in[-T,T],
$$
where $A_t$ is the transport of $A$ by the solution map :
$$
A_t=
\big\{u(t), \,\, {\rm solution\, of\,  \eqref{NLS_focus}\, with\, datum}\,  u_0\, {\rm in }\, A
\big\}
\,.
$$
\end{theoreme}
As already alluded to, whether local in time quasi-invariance holds along focusing problems was raised by Bourgain in \cite[page~28]{Bo_proceeding}.  
\\

Let us finally mention again that  for  sake of simplicity we stated our results only for the nonlinear interaction $|u|^4 u$. 
However, these results may be extended to more general nonlinearities of type $F(|u|^2)u$, $F\geq 0$ ; in particular, $F(|u|^{2})u=|u|^{2k+4} u$, $k \in \mathbb{N}$ requires only routine modifications of our arguments.
\\

The remaining part of the paper is organised as follows. In the next section we show how  Theorem~\ref{th2} can be reduced to energy estimates. 
Then in Section~\ref{sect3}, we prove Theorem~\ref{determinist}. The last section is devoted to the proof of Theorem~\ref{th3}
\\

{\bf Acknowledgement.} We are grateful to Tadahiro Oh for suggesting that we consider the focusing case and for useful remarks on an earlier draft of the manuscript. We also warmly thank the referees for their remarks which helped us to improve the overall presentation of our results.
\section{Reduction to  energy estimates}
In this section, we outline how  Theorem~\ref{determinist} implies Theorem~\ref{th2}.  First, define the following measures: 
$$d\rho_{2k}=f(u) d\mu_{2k}, \quad d\rho_{2k,M}=f_M(u)d\mu_{2k}$$
where, 
$$
f(u)=\exp\big(E_{2k}(u)-\|u\|_{H^{2k}}^2\big)=\exp(R_{2k}(u)),\quad f_M(u)=f(\pi_M u)
$$ 
and $E_{2k}(u)$ is the energy from Theorem~\ref{determinist}.
Next, we shall also use the following representation
$$ d\rho_{2k,M}=\gamma_M \exp{(-E_{2k}(\pi_M u))} du_1\dots du_M \times d\mu_{2k,M}^\perp$$ 
where $\gamma_M$ is a suitable renormalization constant, $du_1\dots du_M$
is the Lebesgue measure on $\C^M$ and $d\mu_{2k,M}^\perp$
is the gaussian measure induced by the random series (compare with \eqref{omeg})
\begin{equation*}
\omega\longmapsto \sum_{n\in\Z, |n|>M}\frac{g_n(\omega)}{(1+n^2)^{k}}\, e^{inx} \,.
\end{equation*}
In the sequel we shall use the following notation
\begin{equation}\label{biar}
B_R:=\Big \{u\in H^{2k-\frac 12 - \epsilon} \,|\,  \|u\|_{H^{2k-\frac 12 - \epsilon}}<R
\Big \}
\end{equation}
where $\epsilon>0$ is suitable fixed number (small enough) in such a way that
the inequality \eqref{NEW_pak_again} is satisfied.
We shall need the following approximation property.
\begin{proposition}\label{appp}
For every $R>0$,
\begin{equation}\label{approx}
\lim_{M\rightarrow \infty}\int_{B_R} |f_M(u) - f(u)|d\mu_{2k}(u)=0.
\end{equation}
\end{proposition}



\begin{proof}
We consider $B_R$ equipped with the measure $\mu_{2k}$. 
Thanks to \eqref{NEW}, we have that  $R_{2k}(\pi_M u)$ converges to $R_{2k}(u)$ in $L^2(B_R)$ and as a consequence 
$R_{2k}(\pi_M u)$ converges to $R_{2k}(u)$ in measure, i.e. 
$$
\lim_{M\rightarrow\infty} \mu_{2k}
\{
u\in B_R\, |\,  |R_{2k}(\pi_M u)-R_{2k}( u)|>\gamma
\}=0,\quad \forall\, \gamma>0\,.
$$
As  a consequence $f_{M}(u)$ converges to $f(u)$ in measure too, i.e. 
\begin{equation}\label{q1}
\lim_{M\rightarrow\infty} \mu_{2k}
\{
u\in B_R\, |\,  |f_{M}(u)-f( u)|>\gamma
\}=0,\quad \forall\, \gamma>0\,.
\end{equation}
Next, thanks to \eqref{NEW}, we have that there is a constant $C>0$ depending on $R$ such that
\begin{equation}\label{q2}
\int_{B_R}|f_M(u)|^2 d\mu_{2k}(u)+\int_{B_R} |f(u)|^2 d\mu_{2k}(u)<C,\quad \forall\, M\in\N.
\end{equation}
Let us now turn to the proof of \eqref{approx}. 
Fix $\gamma>0$ and
set 
$$
A_{M,\gamma}:=\Big \{u\in B_R\, |\,  |f_{M}(u)-f( u)|>\frac{\gamma}{2}\Big \}\,.
$$
We can write
$$
\int_{B_R} |f_M(u) - f(u)|d\mu_{2k}(u)<\frac{\gamma}{2}+\int_{A_{M,\gamma}}|f_M(u) - f(u)|d\mu_{2k}(u)\,.
$$
Using \eqref{q2} and the Cauchy-Schwarz inequality, we can write
$$
\int_{A_{M,\gamma}}|f_M(u) - f(u)|d\mu_{2k}(u)< C\big(\mu_{2k}(A_{M,\gamma})\big)^{\frac{1}{2}}\,
$$
where $C>0$ is another constant that depends on $R$.
Thanks to \eqref{q1} there is $M_0$ such that for $M\geq M_0$,
$$
\big(\mu_{2k}(A_{M,\gamma})\big)^{\frac{1}{2}}<\frac{\gamma}{2C}\,.
$$
This completes the proof of Proposition~\ref{appp}.
\end{proof}
As a consequence of the global in time analysis for the Cauchy problem, we have the following statement.
\begin{proposition}\label{cauchy-sigma}
Let $t>0$.
For every $R>0$ there is $C(R)>0$ such that 
$$
 \Phi_M(s)(B_R)\subset B_{C(R)}, \quad \forall\, s\in [0, t],\quad \forall\, M\in \N\cup \{\infty\}\,.
$$
\end{proposition}
The proof of Proposition~\ref{cauchy-sigma} can be done exactly as in  \cite[Proposition~2.6]{Tz}.
In the sequel  the following proposition will be useful.
\begin{proposition}\label{sigma}
Let $A\subset H^{2k-\frac 12 -\epsilon}$ be a compact set, then for every $\delta>0$ there exists $M_0\in \N$ such that for every $M\geq M_0$,
$$
\Phi(s) (A)\subset \Phi_M(s) (A+B_\delta), \quad \forall s\in [0, t]\, .
$$
\end{proposition}
The proof of Proposition~\ref{sigma} can be done exactly as in  \cite[Proposition~2.10]{Tz}.
\\
\begin{proof}[Proof of Theorem~\ref{th2}]
Fix $t>0$. We can conclude the proof of Theorem~\ref{th2} provided that we show the following implication 
$$A\subset H^{2k-\frac 12 -\epsilon} \hbox{ compact and }
\mu_{2k}(A)=0 \Longrightarrow \mu_{2k}(\Phi(t)(A))=0.$$
In fact once the quasi-invariance is established for any compact set $A$ then, by a classical
argument of measure theory, the same property can be extended to every
measurable set.
By elementary considerations the above implication is equivalent
to the following one
$$A\subset H^{2k-\frac 12 -\epsilon} \hbox{ compact and }
\rho_{2k}(A)=0 \Longrightarrow \rho_{2k}(\Phi(t)(A))=0.
$$

First of all by compactness we can fix $R>0$ such that
$A\subset B_R$  and by Proposition~\ref{cauchy-sigma}
we have
\begin{equation}\label{2R}
\Phi(s) (B_{2R}) \cup \Phi_M(s)(B_{2R})\subset B_{C(R)}, \quad \forall s\in [0, t]
\end{equation}
for a suitable constant $C(R)>0$.
Next, by Liouville theorem  and invariance of complex gaussians by rotations, we get for $D\subset B_{2R}$ a generic measurable set :
\begin{eqnarray*}
\big| \frac d{ds} \rho_{2k,M}(\Phi_M(s)(D))\big | & = &  \big| \frac d{ds} \int_{\Phi_M(s)(D)} f_M(u) d\mu_{2k}(u) \big |
\\
&= &\gamma_M  \big| \frac d{ds} \int_{\Phi_M(s)(D)} \exp{(-E_{2k}(\pi_M u))} du_1\dots du_M \times d\mu_{2k,M}^\perp\big|
\\
&= & \gamma_M \big|\int_{D} \frac d{ds} \exp{(-E_{2k}(\Phi_M(s)(\pi_M u)))} du_1\dots du_M \times d\mu_{2k,M}^\perp \big|\,.
\end{eqnarray*}
Now we use  \eqref{modifene}  of Theorem~~\ref{determinist} 
(along with the estimate \eqref{NEW_pak_again}) and  Proposition~\ref{cauchy-sigma} to get the following key bound: 
\begin{multline*}
 \forall \,s\in [0, t],\, \forall \,u\in D\,,\,\,\,
|\frac d{ds} \exp{(-E_{2k}(\Phi_M(s)(\pi_M u)))}|\nonumber \leq C
 \exp{(-E_{2k}(\Phi_M(s)(\pi_M u)))}
\end{multline*}
where $C>0$ depends on $R$.
Consequently,  for $s\in[0,t]$, we have the bound 
\begin{multline*}
\big| \frac d{ds} d\rho_{2k,M}(
\Phi_M(s)(D))\big | \\
 \leq 
 \gamma_M C  \int_{D} \exp{(-E_{2k}(\Phi_M(s)(\pi_M u)))} du_1\dots du_M \times d\mu_{2k,M}^\perp 
 \\ = C\int_{\Phi_M(s)(D)} f_M(u)d\mu_{2k}(u) = 
C
\rho_{2k,M}(\Phi_M(s)(D)).
\end{multline*}
By using Gronwall lemma we get for every $M\in\N$ and every $s\in [0, t]$,
\begin{equation}\label{zvezda}
\rho_{2k,M}(\Phi_M(s)(D)) = \int_{\Phi_M(s)(D)} f_M(u) d\mu_{2k}(u)
\leq \exp (Cs) \rho_{2k,M}(D) \,. 
\end{equation}
Next,  notice that by Proposition~\ref{sigma} we have
\begin{equation}\label{pizcer}
\rho_{2k}(\Phi(t)(A))=\int_{\Phi(t)(A)} f(u)d\mu_{2k}(u)\leq \int_{\Phi_M(t) (A+B_\delta )} f(u)d\mu_{2k} (u)
\end{equation}
for any $\delta>0$ fixed, provided that $M$ is large. 
Moreover, thanks to Proposition~\ref{appp} we have 
\begin{equation}\label{dil}
\lim_{M\rightarrow \infty}\int_{B_{2R}} |f_M(u)-f(u)|d\mu_{2k}(u)=0 
\end{equation}
and hence by choosing $M$ large enough we can continue  \eqref{pizcer} as follows
$$
\rho_{2k}(\Phi(t)(A))  \leq  \int_{\Phi_M(t) (A+B_\delta)} f_M(u)d\mu_{2k}(u) +\delta \,.
$$
Now, we apply \eqref{zvezda} with $D=A+B_\delta$ (for $\delta<R$). This yields 
$$
\rho_{2k}(\Phi(t)(A)) \leq  \exp (Ct)\int_{A+B_{\delta}} f_M(u) d\mu_{2k}(u) +\delta\,.
$$
Coming back to \eqref{dil}, we obtain that for $M$ large enough
$$
\rho_{2k}(\Phi(t)(A))\leq
 \exp (Ct)\int_{A+B_{\delta}} f(u) d\mu_{2k}(u) +2\delta\,.
$$
Passing to the limit $\delta\rightarrow 0$ and since $A$ is compact, we get
$$
\lim_{\delta\rightarrow 0}\int_{A+B_{\delta}} f(u) d\mu_{2k}(u) =\int_{A} f(u) d\mu_{2k}(u)=\rho_{2k}(A)=0.
$$
Therefore  $\rho_{2k}(\Phi(t)(A))=0$.
This completes the proof of  Theorem~\ref{th2} (assuming that Theorem~\ref{determinist} holds true).
\end{proof}
\section{Proof of Theorem~\ref{determinist}}\label{sect3}

Using that the flow $\Phi_M(t)$ is continuous on $H^{s}(\T)$ for every $s\geq 1$
and a density argument, 
we may suppose from now on that the initial condition $u_0$ is in $C^\infty(\T)$,
and hence $\Phi_M(t)(u_0)\in C^\infty(\R\times \T)$.\\

In order to better highlight the key points in the argument, we first prove Theorem~\ref{determinist} in the case $k=1$, $M=\infty$. In a second step we generalize it to the case $k>1$, $M=\infty$. Finally, we show how to treat the case $M<\infty$.\\

From now on, for clarity's sake, we shall use the notation $\lesssim$ in
order to denote a lesser or equal sign $\leq$ up to a positive multiplicative constant C, that in turn
may depend harmlessly on contextual parameters. 

\subsection{Proof of Theorem~\ref{determinist} for $k=1$, $M=\infty$}\label{k=1}
We first consider the case $k=1$, $M=\infty$, namely we prove the estimate  for the flow $\Phi(t)(u_0)$.  
We denote by $\tilde{\mathcal L}_{2,k_0}$ and $\tilde{\mathcal P}_{2,k_0}$ the 
sets of functionals defined as follows:
\begin{multline}\label{L2}
\tilde {\mathcal L}_{2,k_0}:=\Big\{ F(u(t,x))\,\, |  \,\ u(t,x)\in C^\infty(\R\times \T),
\,\,\, F: C^\infty(\T)\rightarrow \R,\\|F(u)|\lesssim \big( 1+ \|u\|_{H^{1}}^{k_0}\big)\big(1+\|\partial_x u\|^4_{L^4}\big)\Big\}
\end{multline}
and
\begin{multline}\label{P2}
\tilde {\mathcal P}_{2, k_0}:=\Big\{\frac d{dt} G(u(t,x))  \,\,| \,\,
u(t,x)\in C^\infty(\R\times \T),\,\,\, G: C^\infty(\T)\rightarrow \R, \,\, \\
G(0)=0,\,\,\,
|G(u)-G(v)|\lesssim \|u-v\|_{H^{1}}\big(1+\|u\|_{H^{1}}^{k_0} +\|v\|_{H^{1}}^{k_0} \big)
\Big\}.
\end{multline}
We introduce a useful  notation for a given couple of functionals $F_1, F_2:C^\infty(\T)\rightarrow \R$:
\begin{multline}\label{equivk1}
F_1(u(t,x))\equiv F_2(u(t,x)) \iff \\F_1(u(t,x))-F_2(u(t,x))=F(u(t,x))+\frac d{dt} G(u(t,x)), \\
 \quad \hbox{ where } \quad F\in \tilde{\mathcal L}_{2,k_0}, \quad
\frac d{dt} G(u(t,x))\in \tilde{\mathcal P}_{2,k_0} \hbox{ for some }  k_0\in\N.
\end{multline}
Let $u(t,x)$ be a smooth solution of \eqref{NLS}.  
With the notation $\equiv$ introduced in \eqref{equivk1}, our statement is equivalent to proving that
$
\frac d{dt} \|u\|_{H^{2}}^2\equiv 0.
$
We compute, using the equation solved by $u$:
$$
\frac d{dt} \|\partial_x^2  u\|_{L^2}^2= 2 {\Re} (\partial_x^2 \partial_t  u, \partial_x^2  u)=2 {\Re} (\partial_x^2 \partial_t  u, -i\partial_t  u +  u| u|^4)\,
$$
where $(.,.)$ is the usual $L^2$ scalar product.
The contribution of $-i\partial_t  u$ vanishes by integration by parts. 
Observe that the defocusing nature of the equation is not of importance. 
Therefore,
\begin{align*}\label{vv1}
\frac d{dt} \|\partial_x^2  u\|_{L^2}^2 & = 2 {\Re} (\partial_t  \partial_x^2  u,   u| u|^4)\\
  & =2 \frac d{dt} {\Re} (\partial_x^2 u, u|u|^4)- 2{\Re} \big(\partial_x^2 u, \partial_t (u|u|^4)\big)\,.
\end{align*}
By integration by parts and Sobolev embedding $H^1\subset L^\infty$ we get, 
$$
|{\Re} (\partial_x^2 u, u|u|^4)-{\Re} (\partial_x^2 v, v|v|^4)|\lesssim \|u-v\|_{H^1}(\|u\|_{H^1}^5+\|v\|_{H^1}^5)
$$
and in particular 
\begin{equation}\label{gif}\frac d{dt} {\Re} (\partial_x^2 u, u|u|^4)\in \tilde{\mathcal P}_{2,5}.
\end{equation}
Then we get 
\begin{equation}\label{forum0}
\frac d{dt} \|\partial_x^2 u\|_{L^2}^2\equiv - 2{\Re} \big(\partial_x^2 u, \partial_t (u|u|^4)\big)\,.
\end{equation}
Next, notice that
\begin{align}\label{gifyvette}
{\Re} \big(\partial_x^2 u, \partial_t (u|u|^4)\big) &=   {\Re} (\partial_x^2 u, \partial_t u |u|^4) + {\Re} \big(\partial_x^2 u, u\partial_t (|u|^4)\big)
\\\nonumber
 &=  {\Re} (\partial_x^2 u, \partial_t u |u|^4) + \frac 12 \big(\partial_x^2 |u|^2, \partial_t (|u|^4)\big) - \big(|\partial_x u|^2, \partial_t (|u|^4)\big) \,.
\end{align}
Using that $u$ solves \eqref{NLS}, we get
$$
 {\Re} (\partial_x^2 u, \partial_t u |u|^4)=  {\Re} (\partial_x^2 u , -i |u|^8 u)={\Im} \big(\partial_x u, u \partial_x (|u|^{8})\big)\equiv 0,
$$
where in the last step we used again Sobolev embedding $H^1\subset L^\infty$
in order to get
$
\big|{\Im} \big(\partial_x u, u \partial_x (|u|^{8})\big)\big|\lesssim \|u\|_{H^1}^{10}\,
$
and hence
\begin{equation}\label{yvette}{\Im} \big(\partial_x u, u \partial_x (|u|^{8})\big)\in \tilde{\mathcal L}_{2,10}.\end{equation}
Therefore going back to \eqref{gifyvette} we get
\begin{equation}\label{forum}
{\Re} \big(\partial_x^2 u, \partial_t (u|u|^4)\big) 
\equiv
 \frac 12 \big(\partial_x^2 |u|^2, \partial_t (|u|^4)\big) - \big(|\partial_x u|^2, \partial_t (|u|^4)\big) \,.
\end{equation}
Next, we introduce the mass density 
\begin{equation}\label{defN}N:=|u|^{2}
\end{equation} and we notice that we can write the first term on the r.h.s. in \eqref{forum} as follows:
\begin{eqnarray*}
\frac 12 \big(\partial_x^2 N, \partial_t (N^2)\big)&= &-\frac 12 \big(\partial_x N, \partial_t \partial_x(N^2)\big) = 
- \big(\partial_x N, \partial_t (\partial_x N N)\big)
\\\
& = & - \frac d{dt} \big(\partial_x N, \partial_x N N\big)
+   \big( \partial_t \partial_x N, \partial_x N N\big)
\\
&=& -\frac d{dt} \big(\partial_x N, \partial_x N N\big)
+ \frac{1}2\big( \partial_t (\partial_x N)^2, N\big)
\\
&= & -\frac d{dt} \big(\partial_x N, \partial_x N N\big)+ \frac{1}2 \frac d{dt}  \big((\partial_x N)^2, N\big)-\frac{1}2  \big((\partial_x N)^2, \partial_t N\big)  
\\
&= & - \frac{1}2 \frac d{dt}  \big((\partial_x N)^2, N\big)-\frac{1}2  \big((\partial_x N)^2, \partial_t N\big)  \,.
\end{eqnarray*}
Arguing as along the proof of \eqref{gif} we get
\begin{equation}\label{easP}
\frac d{dt} ((\partial_x N)^2, N)=\frac d{dt} ((\partial_x |u|^2)^2, |u|^2)\in \tilde {\mathcal P}_{2,5},
\end{equation}
and hence
\begin{equation}\label{K1=K2}
\frac 12 \big(\partial_x^2 N, \partial_t (N^2)\big)\equiv -\frac{1}2  \big((\partial_x N)^2, \partial_t N\big)  \,.
\end{equation}
Next, by combining \eqref{forum0}, \eqref{forum} and \eqref{K1=K2}, we get
\begin{equation}\label{forum2}
\frac d{dt} \|\partial_x^2 u\|_{L^2}^2 \equiv \big((\partial_x N )^2, \partial_t N\big)+2 \big(|\partial_x u|^2, \partial_t (N^2)\big)
\,.
\end{equation}
Introducing the momentum density 
\begin{equation}\label{defJ}J:=2\Im (\bar u \partial_{x} u),
\end{equation} one easily checks that
\begin{equation}\label{eleele}
J^2+(\partial_x N)^2=4 N |\partial_x u|^2,
\end{equation}
where $N$ is given by \eqref{defN},
and therefore 
by \eqref{forum2}
\begin{equation}\label{modified}
\frac d{dt} \|\partial_x^2 u\|_{L^2}^2 \equiv 2\big( \partial_t N,  (\partial_x N )^2\big)
+
\big( \partial_t N, J^2\big)
\,.
\end{equation}
Notice that if we blindly substitute $\partial_t u$, using equation \eqref{NLS}, we still get a term $\partial_x^2 u$ that we precisely seek to avoid.
Precisely, the goal will be to split this second derivative that we would get by substitution into two first order derivatives hence avoiding second order derivatives on the r.h.s.
This analysis goes much beyond our previous work \cite{PTV}, and it will crucially rely on the following two conservation laws at the density level. First we introduce along with $N$  and $J$ the stress-energy tensor
\begin{equation}\label{defT}
T:=4|\partial_{x} u|^{2}-\partial_x^2  N +\frac{4}{3} N^3\,,
\end{equation}
then we have the two following key identities, which provide, when integrated, the mass and momentum conservations.
\begin{lemme}\label{cle}
The following identities hold 
\begin{equation}\label{contin}
\partial_{t} N+\partial_{x} J=\partial_{t} J+\partial_{x} T=0\,.
\end{equation}
\end{lemme}
In view of  \eqref{modified}, we set 
\begin{equation}\label{defN1J0}
N_1:=\int \partial_t N (\partial_x N )^2, \quad
J_0:=\int \partial_t N J^2\,,
\end{equation}
so that we have 
\begin{equation}\label{biagtibe}
\frac d{dt} \|\partial_x^2 u\|_{L^2}^2 \equiv 2N_{1}+J_{0}\,.
\end{equation}
\begin{lemme} \label{cergy}
We have the following relations 
\begin{equation}\label{tatifr}
2 N_1+ \frac 12 J_0\equiv 0,\quad  J_0= 0\,.
\end{equation}
As a consequence $ N_1\equiv J_0\equiv 0$. 
\end{lemme}
\begin{proof}
First notice that by \eqref{contin} we write 
$$J_0=\int \partial_t N J^2=- \int  \partial_x J J^2
=0\,
$$
and hence we get the second equivalence in \eqref{tatifr}.

Next from \eqref{eleele}, using Lemma~\ref{cle} and the definition
of $T$ (see \eqref{defT}), we can write
\begin{eqnarray*}
N_1+ J_0 & = & \int \partial_t N (\partial_x N)^2+  \int \partial_t N J^2
\\
& = &\int\partial_t N N (T+\partial_{x}^{2} N-\frac 43 N^3)\\
& \equiv & - \int \partial_x J N T+\frac 12 \int \partial_t (N^2) \partial_x^2 N.
\end{eqnarray*}
where at the last step we have used 
$\int \partial_t N N^4=\frac 15 \frac d{dt} \int N^5=
\frac 15 \frac d{dt} \int |u|^{10} \in \tilde {\mathcal P}_{2, 9}$, that in turn follows 
by an argument similar to the one used along the proof of \eqref{gif}.
Recalling \eqref{K1=K2} we can continue
\begin{equation}\label{forum4}
N_1+ J_0\equiv - \int \partial_x J N T-\frac 12 N_1\,.
\end{equation}
Next, another use of  Lemma~\ref{cle} yields 
\begin{equation}\label{forum5}
 - \int \partial_x J N T= \int J\partial_x N T + \int J N  \partial_x T
 =
  \int J\partial_x N T - \int J N  \partial_t J\,.
\end{equation}
Coming back to the definition of $T$,  and using once more  Lemma~\ref{cle}  we get
\begin{align}\label{forum12}
  \int J\partial_x N T = -
\int J\partial_x N \partial_x^2 N
+4 \int J\partial_x N |\partial_x u|^2  + \frac 43 \int J\partial_x N N^3
\\\nonumber \equiv 
  -
\int J\partial_x N \partial_x^2 N
=
\frac 12 \int \partial_x J (\partial_x N)^2=-\frac{1}{2}N_1 \,
\end{align}
where we used $\int J\partial_x N |\partial_x u|^2, \int J\partial_x N N^3\in 
\tilde {\mathcal L}_{2, k_0}$ for a suitable $k_0$ (the proof is similar to the one of \eqref{yvette}).
Next, we have 
\begin{equation}\label{forum6}
- \int J N  \partial_t J=\frac{1}{2}\int \partial_t N J^2 
-\frac 12 \frac d{dt} \int NJ^2\equiv \frac{1}{2}\int \partial_t N J^2=\frac{1}{2} J_0\,,
\end{equation}
where we used
$\frac d{dt} \int NJ^2\in \tilde {\mathcal P}_{2,5}$
(see the proof of \eqref{gif}).
Therefore summarizing by \eqref{forum4}, \eqref{forum5}, \eqref{forum12} and
\eqref{forum6} we get
$$
N_1+ J_0\equiv \frac 12 J_0 -\frac{1}{2} N_1-\frac{1}{2}N_1
=\frac{1}{2}J_0-N_1
\,.
$$
This completes the proof of Lemma~\ref{cergy}.
\end{proof}
Therefore by combining \eqref{biagtibe} and Lemma \ref{cergy}
we get 
\begin{equation}\label{nifile}\frac d{dt} \|\partial_x^2 u\|_{L^2}^2\equiv 0.
\end{equation} Next notice that
$$
\big|\|(1-\partial_x^2)u\|_{L^2}^2-\|\partial_x^2 u\|_{L^2}^2\big|
\lesssim \|u\|_{H^1}^2
$$
and therefore it is easy to deduce 
$$\frac d{dt} \big(\|(1-\partial_x^2)u\|_{L^2}^2-\|\partial_x^2 u\|_{L^2}^2\big)\in 
\tilde{\mathcal P}_{2,1}.$$
By combining this fact with \eqref{nifile} we get $\frac d{dt} \|u\|_{H^2}^2\equiv 0$.
This completes the proof of 
Theorem~\ref{determinist} in the case  $k=1$, $M=\infty$. 
\begin{rem}\label{coloh}
{\rm 
The equivalence $\frac d{dt} \|u\|_{H^2}^2\equiv 0$ and the simple $L^4$ Strichartz estimate for the periodic Schr\"odinger equation imply 
that the solutions of \eqref{NLS} satisfy
$$
\frac{d}{dt}\|u(t)\|_{H^2}^2\leq C(\|u(0)\|_{H^1})\,.
$$
This estimate implies that the $H^2$ norm of the solutions of \eqref{NLS} are bounded by $C|t|^{1/2}$ for $t\gg 1$. 
This gives an alternative (and in our opinion simpler) proof of a result obtained in \cite{Bo3,CKO}, avoiding altogether  normal form and multilinear estimates techniques. We will address elsewhere higher order Sobolev norms, which may be handled similarly by refining our forthcoming analysis on higher order modified energies.}
\end{rem}
\subsection{Proof of Theorem~\ref{determinist} for $k\geq 2$, $M=\infty$}
We denote by $\mathcal L_{2k,k_0}$ and $\mathcal P_{2k,k_0}$ the 
sets of functionals defined as follows:
\begin{multline}\label{lstorto}
{\mathcal L}_{2k,k_0}:=\Big\{F(u(t,x))\,\, | \,\, u(t,x)\in C^\infty(\R\times \T),\,\,\, \,\\F: C^\infty(\T)\rightarrow \R, \,|F(u)|\lesssim \big(1+\|u\|_{H^{2k-1}}^{k_0}\big) \Big\}
\end{multline}
and
\begin{multline}\label{pstorto}
{\mathcal P}_{2k, k_0}:=\Big\{\frac d{dt} G(u(t,x))  \,\,| \,\,
u(t,x)\in C^\infty(\R\times \T),\,\,\, G: C^\infty(\T)\rightarrow \R, \,\,\\
G(0)=0,\,\,\,
|G(u)-G(v)|\lesssim \|u-v\|_{H^{2k-1}}\big(1+\|u\|_{H^{2k-1}}^{k_0} +\|v\|_{H^{2k-1}}^{k_0} \big)
\Big\}\,.
\end{multline}
It is worth mentioning that we have the following inclusions: 
$${\mathcal P}_{2j, k_0}\subseteq {\mathcal P}_{2k, k_0}, \quad
{\mathcal L}_{2j, k_0}
\subseteq
{\mathcal L}_{2k, k_0}, \quad j\leq k$$ and 
$${\mathcal P}_{2k, k_1}\subseteq {\mathcal P}_{2k, k_2}, \quad k_1\leq k_2.$$

We next introduce the following notation. Let $k$ be fixed, then
for a given couple of functionals $F_1, F_2:C^\infty(\T)\rightarrow \R$ we say
\begin{multline}\label{equivk2}
F_1(u(t,x))\equiv F_2(u(t,x)) \iff \\
F_1(u(t,x))-F_2(u(t,x))=F(u(t,x))+\frac d{dt} G(u(t,x)),
\\ \hbox{ where }
\quad F(u(t,x))\in {\mathcal L}_{2k,k_0}, \quad \frac d{dt} G(u(t,x))\in {\mathcal P}_{2k,k_0} 
\hbox{ with a suitable } \quad k_0\in\N.\end{multline}
\begin{rem}
{\rm
Notice that, for $k=1$, the functionals in
 \eqref{pstorto} define the same family as \eqref{P2}.
 However, the family defined in \eqref{lstorto} for $k=1$ and the one
 defined
in  \eqref{L2} differ slightly: notice that we have,
 by using the Sobolev embedding  the $H^\frac 12\subset L^4$,
the
 following inclusion
 \begin{multline*}\tilde {\mathcal L}_{2,k_0} 
\subseteq {\mathcal L}^\epsilon_{2,k_0+4}:=  \Big\{F(u(t,x))\,\, | \,\, u(t,x)\in C^\infty(\R\times \T),\,\,\, \,\\F: C^\infty(\T)\rightarrow \R,\,\,   \,\,|F(u)|\lesssim \big(1+\|u\|_{H^{\frac 32 -\epsilon}}^{4+k_0}\big) \Big\}
\end{multline*} for a suitable $\epsilon>0$.
On the other hand we have the following trivial inclusion
\begin{multline*}{\mathcal L}_{2k,k_0} 
\subseteq {\mathcal L}^\epsilon_{2k,k_0}:=  \Big\{ F(u(t,x))\,\, | \,\, u(t,x)\in C^\infty(\R\times \T),\,\,\, \,\\F: C^\infty(\T)\rightarrow \R, \,\,  \,\,|F(u)|\lesssim \big(1+\|u\|_{H^{2k-\frac 12 -\epsilon}}^{k_0}\big) \Big\}
\end{multline*}
for a suitable $\epsilon>0$. It is worth mentioning that in order to get Theorem \ref{th2}
from Theorem \ref{determinist} it is sufficient  to show that the functionals
$F_{2k}^{(M)}(u)$ belong to ${\mathcal L}^\epsilon_{2k,m_0}$ (defined above) for a suitable $\epsilon>0$ and for some $m_0\in \N$ (see \eqref{NEW_pak_again}). Hence we
could work in a unified framework by using the class 
${\mathcal L}^\epsilon_{2k,k_0}$
for every $k\geq 1$ (without any distinction between $k=1$ and $k>1$).
However our motivation to deal separately with the case $k=1$ and $k\geq 2$ 
along the proof of Theorem \ref{determinist} is twofold:
on one hand the proof for $k=1$ is less involved, compared with the case $k\geq 2$, and hence it is very useful to get the basic ideas behind our argument; on the other hand, 
the smoothing estimate \eqref{NEW_pak_2} implies a strong information on the growth 
of the $H^2$ norm (see remark \ref{coloh}) that in principle we cannot get if we replace
on the r.h.s. in \eqref{NEW_pak_2} the norm $\|\partial_x u\|_{L^4}$ by the stronger norm $\|u\|_{H^{\frac 32 -\epsilon}}$
for $\epsilon>0$ small.}


\end{rem}

In the sequel we introduce some specific families of functionals $F(u)$ 
belonging to ${\mathcal L}_{2k,k_0}$, that will play a crucial role.
The first interesting family is the following one:
\begin{equation}\label{family1}
\Big \{F(u)=\int \prod_{i=1}^{k_0} \partial_x^{\alpha_i} v_i \,\, |  \,\, v_i\in\{u, \bar u\}, \hbox{ }\sum_{i=1}^{k_0}\alpha_i=4k, \hbox{ } {\mathcal Card} \{\alpha_i>1\}\geq 3\Big \}\subseteq {\mathcal L}_{2k, k_0}.
\end{equation}
We shall need the following class as well
\begin{equation}\label{family2}
\Big \{F(u)=\int \prod_{i=1}^{k_0} \partial_x^{\alpha_i} v_i \,\, |  \,\, v_i\in\{u, \bar u\}, 
\hbox{ } \sum_{i=1}^{k_0}\alpha_i=4k, \hbox{ } {\mathcal Card} \{\alpha_i>0\}\geq 4\Big \}\subseteq {\mathcal L}_{2k, k_0},
\end{equation}
and finally 
\begin{equation}\label{family3}
\Big \{F(u)=\int \prod_{i=1}^{k_0} \partial_x^{\alpha_i} v_i \,\, |  \,\, v_i\in\{u, \bar u\}, \hbox{ } \sum_{i=1}^{k_0}\alpha_i\leq 4k-2\Big \}\subseteq {\mathcal L}_{2k, k_0}.
\end{equation}
The proof of the inclusions \eqref{family1}, \eqref{family2}, \eqref{family3} are similar. 
For instance let us sketch the proof of \eqref{family1}.
It is not restrictive to assume
$\alpha_i\geq \alpha_{i+1}$. Then we have two possibilities: 
\begin{itemize}
\item ${\alpha_i}\leq 2k-1, \quad \forall i=1,\dots , k_0$;
\item $\alpha_1\geq 2k$
\end{itemize}
Notice that in the first case, by the constraint $\sum_{i=1}^{k_0}\alpha_i=4k 
$ and ${\mathcal Card} \{\alpha_i>1\}\geq 3$ we have that necessarily $
\max\{\alpha_3,\dots , \alpha_{k_0}\}=\alpha_3\leq 2k-2$. Hence
by combining the H\"older inequality and the Sobolev embedding $H^1\subset L^\infty$ we get
\begin{multline}\label{family1proof}
|\int \prod_{i=1}^{k_0} \partial_x^{\alpha_i} v_i|\leq \|\partial_x^{\alpha_1} v_1\|_{L^2}
\|\partial_x^{\alpha_2} v_2\|_{L^2} \prod_{i=3}^{k_0} \|\partial_x^{\alpha_i} v_i\|_{L^\infty}
\\
\lesssim \|v_1\|_{H^{2k-1}}
\|v_2\|_{H^{2k-1}} \prod_{i=3}^{k_0} \|\partial_x^{\alpha_i} v_i\|_{H^{1}}
\\\lesssim \|v_1\|_{H^{2k-1}}
\|v_2\|_{H^{2k-1}} \prod_{i=3}^{k_0} \|v_i\|_{H^{\alpha_i+1}}
\lesssim\|u\|_{H^{2k-1}}^{k_0}.
\end{multline}
In the second case, namely $\alpha_1\geq 2k$, notice that by the constraint
$\sum_{i=1}^{k_0}\alpha_i=4k$ and ${\mathcal Card} \{\alpha_i>1\}\geq 3$
we have necessarily $\alpha_2, \alpha_3\geq 2$ and $\max\{\alpha_2, \dots, \alpha_{k_0}\}=\alpha_2\leq 2k-2$. Next by integration by parts
(that we can apply $\alpha_1-2k$ times), we reduce 
$\int \prod_{i=1}^{k_0} \partial_x^{\alpha_i} v_i $ to a linear combination of expressions 
of the following type:
\begin{equation}\label{biagjes}\int \prod_{i=1}^{k_0} \partial_x^{\beta_i} v_i\end{equation}
where again we can assume $\beta_i\geq \beta_{i+1}$ and $\beta_1=2k$.
Hence by recalling the conditions on $\alpha_i$, necessarily we have
$\sum_{i=1}^{k_0} \beta_i=4k$, $\beta_2, \beta_3\geq 2$ and hence
$\max\{\beta_2,\dots , \beta_{k_0}\}=\beta_2\leq 2k-2$.
By one last integration by parts (that moves one derivative from $\partial_x^{\beta_1}v_1$
on the other factors in \eqref{biagjes}) we can conclude by combining H\"older inequality and the Sobolev embedding $H^1\subset L^\infty$, in the same spirit
as in \eqref{family1proof}.
\\
\\

Now we can focus on the proof of Theorem \ref{determinist} for $k\geq 2$, $M=\infty$.
Notice that using the notation $\equiv$ we introduced in \eqref{equivk2}, it is 
equivalent to prove that $\frac d{dt} \|u\|_{H^{2k}}^2\equiv 0$.
In fact it is sufficient to show
\begin{equation}\label{homogen}\frac d{dt} \|\partial_x^{2k}u\|_{L^2}^2\equiv 0,\end{equation}
then the stronger conclusion $\frac d{dt} \|u\|_{H^{2k}}^2\equiv 0$ follows 
by an argument similar to the one used at the end of subsection \ref{k=1} (above remark \ref{coloh}).

Let $u(t,x)$ be a smooth solution of \eqref{NLS}.
Then by using the equation solved by $u$ we get:
\begin{align}\label{first}
\frac d{dt} \|\partial_x^{2k} u\|_{L^2}^2&=2 \Re (\partial_t \partial_x^{2k} u, 
\partial_x^{2k} u)\\\nonumber &=2 
\Re \big(-i \partial_t^2 \partial_x^{2k-2} u+ \partial_t \partial_x^{2k-2} (u|u|^4), -i \partial_t \partial_x^{2k-2} u +\partial_x^{2k-2} (u|u|^4)\big)
\\\nonumber
&\equiv 2 \Re (-i \partial_t^2 \partial_x^{2k-2} u, -i \partial_t \partial_x^{2k-2} u)=
\frac d{dt} \|\partial_t \partial_x^{2k-2} u\|_{L^2}^2,
\end{align}
where the equivalence $\equiv$ follows by 
\begin{equation}\label{albfin1}
2 \Re (-i \partial_t^2 \partial_x^{2k-2} u, \partial_x^{2k-2} (u|u|^4)) + 2 \Re (
\partial_t \partial_x^{2k-2} (u|u|^4), -i \partial_t \partial_x^{2k-2} u)\in {\mathcal P}_{2k, k_0}
\end{equation}
for a suitable $k_0$ 
and 
\begin{equation}\label{albfin}
2 \Re (\partial_t \partial_x^{2k-2} (u|u|^4)), \partial_x^{2k-2} (u|u|^4))
\in {\mathcal P}_{2k,9}\,.
\end{equation}
The proof of \eqref{albfin} follows by the following computation
$$2 \Re (\partial_t \partial_x^{2k-2} (u|u|^4)), \partial_x^{2k-2} (u|u|^4))
=\frac d{dt}  \int |\partial_x^{2k-2} (u|u|^4)|^2,$$
hence we have to show
\begin{equation}\label{tarsi1}\frac d{dt}  \int |\partial_x^{2k-2} (u|u|^4)|^2
\in {\mathcal P}_{2k,9}.\end{equation}
The main point is to check the property that describes
${\mathcal P}_{2k,9}$ (see \eqref{pstorto}) for $v=0$, which in turn reduces
to proving that $ \int |\partial_x^{2k-2} (u|u|^4)|^2 \in {\mathcal L}_{2k,10}$.
Indeed once this is proved then it is clear that the same argument, 
with minor changes, allows to deal with the full proof of \eqref{albfin}
including the case the case $v\neq 0$ in the definition \eqref{pstorto}.
Next in order to conclude notice that $\int |\partial_x^{2k-2} (u|u|^4)|^2$
belongs to the family \eqref{family3} for $k_0=10$ and hence we get
\eqref{tarsi1}.
Concerning the proof of \eqref{albfin1} we have
\begin{multline*}
2 \Re (-i \partial_t^2 \partial_x^{2k-2} u, \partial_x^{2k-2} (u|u|^4)) + 2 \Re (
\partial_t \partial_x^{2k-2} (u|u|^4), -i \partial_t \partial_x^{2k-2} u)
\\= 2 \Im (\partial_t^2 \partial_x^{2k-2} u, \partial_x^{2k-2} (u|u|^4)) - 2 \Im (
\partial_t \partial_x^{2k-2} (u|u|^4), \partial_t \partial_x^{2k-2} u) 
\\\,\,\,\,\,\,\,\,=2 \frac d{dt} \Im (\partial_t \partial_x^{2k-2} u, \partial_x^{2k-2} (u|u|^4))
- 2  \Im (\partial_t \partial_x^{2k-2} u, \partial_t \partial_x^{2k-2} (u|u|^4)) 
\\{}- 2 \Im (
\partial_t \partial_x^{2k-2} (u|u|^4)), \partial_t \partial_x^{2k-2} u)\\
 =
2 \frac d{dt} \Im (\partial_t \partial_x^{2k-2} u, \partial_x^{2k-2} (u|u|^4)),
\end{multline*}
hence we get \eqref{albfin1} provided that we show
\begin{equation}\label{tarsi}
\frac d{dt} \Im (\partial_t \partial_x^{2k-2} u, \partial_x^{2k-2} (u|u|^4))
\in {\mathcal P}_{2k,k_0}\end{equation}
for a suitable $k_0$. Arguing as in the proof of \eqref{tarsi1},
it is sufficient to show
\begin{equation}\label{lastissimo}
\Im (\partial_t \partial_x^{2k-2} u, \partial_x^{2k-2} (u|u|^4))
\in {\mathcal L}_{2k,k_0+1}\end{equation} for a suitable $k_0$.
By using the equation solved by $u$ we can replace 
$\partial_t u$ by $i(\partial_x^2 u - u|u|^4)$.
Let us consider the worse case when we replace 
$\partial_t u$ by $i\partial_x^2 u$ (this is the most delicate case since we get the highest amount of derivatives). Using Leibniz rule to expand space derivatives
we get that $\Im (i\partial_x^{2k} u, \partial_x^{2k-2} (u|u|^4))$
can be written as a linear combination of terms of the following type
\begin{equation*}
\int \partial_x^{j_1} v_1\partial_x^{j_2} v_2 \partial_x^{j_3} 
v_3\partial_x^{j_4} v_4 \partial_x^{j_5} v_5 \partial_x^{j_6}v_6\end{equation*}
where $v_i\in \{u, \bar u\}$, $j_1+j_2+j_3+j_4+j_5+j_6=4k-2$ and hence the terms above belong to the family \eqref{family3} for $k_0=6$. As a consequence of this fact, and by noticing that the remaining (lower order) contributions that we get when we replace $\partial_t u$ by $-iu|u|^4$, can be treated in a similar way, 
we get \eqref{lastissimo}.

By iterating $k-1$ times the argument used to get \eqref{first}  (that allows to replace two space derivatives by one time derivative) we can deduce
\begin{equation}\label{guido}\frac d{dt} \|\partial_x^{2k} u\|_{L^2}^2
\equiv 
\frac d{dt} \|\partial_t^{k} u\|_{L^2}^2.\end{equation}
Then we can continue as follows by using the equation solved by $u(t,x)$ and the Leibniz rule:
\begin{align}\label{degi}\frac d{dt} \|i\partial_t^k u\|_{L^2}=2 \Re (i\partial_t^{k+1} u, i\partial_t^k u)
\\\nonumber =2 \Re (\partial_t^k (-\partial_x^2 u + u|u|^{4}) , i\partial_t^k u)
=2 \Re (\partial_t^k (u|u|^{4}) , i\partial_t^k u)\\\nonumber
=2 \Re (\partial_t^k u|u|^{4} , i\partial_t^k u)+2 \Re (u \partial_t^k (|u|^{4}) , i\partial_t^k u)
+\sum_{j=1}^{k-1} c_j \Re (\partial_t^j u \partial_t^{k-j} (|u|^{4}) , i\partial_t^k u)\\
=\underbrace{2 \Re (u \partial_t^k (|u|^{4}) , i\partial_t^k u)}_{I}
+\underbrace{\sum_{j=1}^{k-1} c_j \Re (\partial_t^j u \partial_t^{k-j} (|u|^{4}) , i\partial_t^k u)}_{II}
\end{align}
for suitable coefficients $c_j$.
\\

First we claim that $II\equiv 0$.  
More precisely we shall prove 
\begin{equation}\label{impoterm}
(\partial_t^j u \partial_t^{k-j} (|u|^{4}) , i\partial_t^k u)
\in {\mathcal L}_{2k, k_0}, \quad \forall j=1,\dots , k-1\end{equation}
for a suitable $k_0$.
Notice that when we expand the time derivative by Leibnitz rule,
we can write
$
(\partial_t^j u \partial_t^{k-j} (|u|^{4}) , i\partial_t^k u)
$
as a linear combination of terms of the following type:
\begin{equation}\label{bucc}
\int \partial_t^{j_1} v_1\partial_t^{j_2} v_2 \partial_t^{j_3} 
v_3\partial_t^{j_4} v_4 \partial_t^{j_5} v_5 \partial_t^{j_6}v_6\end{equation}
where $v_j\in \{u, \bar u\}$ and we can assume
\begin{equation}\label{glfiromfi} j_1+j_2+j_3+j_4+j_5+j_6=2k,
\quad j_i\geq j_{i+1}, \quad j_1, j_2, j_3\geq 1 .\end{equation}
By using the equation solved by $u$ we can replace (in \eqref{bucc})
$\partial_t v_i$ by $\pm i(\partial_x^2 v_i - v_i|v_i|^4)$.
Let us consider the worse case where we replace 
$\partial_t v_i$ by $\pm i\partial_x^2 v_i$ (this is the most delicate case since we get the higher
amount of derivatives) and we get
from \eqref{bucc} the following contribution
\begin{equation}\label{bucccontrmain}
\int \partial_x^{2j_1} v_1\partial_x^{2j_2} v_2 \partial_x^{2j_3} 
v_3\partial_x^{2j_4} v_4 \partial_x^{2j_5} v_5 \partial_x^{2j_6}v_6\end{equation}
where 
$$2j_1+2j_2+2j_3+2j_4+2j_5+2j_6=4k, \quad 2j_1, 2j_2, 2j_3\geq 2.$$
Then the typical expression \eqref{bucccontrmain} fits in the family 
\eqref{family1} and hence we conclude that the expression in \eqref{bucccontrmain}
belongs to ${\mathcal L}_{2k,6}$.
Notice that in the case that we replace in \eqref{bucc} at least once
$\partial_t v_i$ by the nonlinear term $v_i|v_i|^4$ (and not by $\partial_x^2 v_i$ as we did above) then we get a multilinear expression of derivatives of $v_i$
(whose homogeneity is higher than six) where the amount of derivatives involved
is less or equal that $4k-2$. Then the terms that we get in this case fit with the family of functionals described by \eqref{family3} and hence belonging to 
${\mathcal L}_{2k,k_0}$. By considering all the possible contributions, we get \eqref{impoterm}.
\\

Concerning the term $I$ in \eqref{degi} we get:
\begin{multline}\label{tzle}
I=2 \Re (u \partial_t^k (|u|^4) , i\partial_t^k u) =
2 \frac d{dt} \Re (u \partial_t^k (|u|^4) , i\partial_t^{k-1} u)
\\
- 2 \Re (\partial_t u \partial_t^k (|u|^4) , i\partial_t^{k-1}  u) - 2 \Re (u \partial_t^{k+1} (|u|^4) , i\partial_t^{k-1} u)\,.
\end{multline}
Arguing as along the proof of \eqref{tarsi} and \eqref{impoterm}, one can prove
$$\frac d{dt} \Re (u \partial_t^k (|u|^4) , i\partial_t^{k-1} u)
\in {\mathcal P}_{2k,k_0}, \quad \Re (\partial_t u \partial_t^k (|u|^4) , i\partial_t^{k-1}  u)\in {\mathcal L}_{2k,k_0}$$
for a suitable $k_0$.
Hence we can continue \eqref{tzle} as follows
\begin{equation}\label{tzbe}I\equiv - 2 \Re (u \partial_t^{k+1} (|u|^4) , i\partial_t^{k-1} u).\end{equation}
Next, notice that
\begin{equation}\label{implor}
\Re (\partial_t u \partial_t^{k+1} (|u|^4), i\partial_t^{k-2} u) \in {\mathcal L}_{2k,k_0},
\quad \hbox{ if } k\geq3
\end{equation}
for a suitable $k_0$,
whose proof is similar to the one of \eqref{impoterm}.
\\
We claim that
\begin{equation}\label{tinavita}I\equiv 2 (- 1)^{k-1}  \Re (u \partial_t^{2k-1} (|u|^4) , i\partial_t u).
\end{equation}
In fact in the case $k=2$ it follows by \eqref{tzbe}. In the general case $k\geq 3$ we argue as follows.
By combining \eqref{tzbe} and \eqref{implor} (that works for $k\geq 3$) we get
\begin{multline}\label{tzbenew}I\equiv - 2 \Re (u \partial_t^{k+1} (|u|^4) , i\partial_t^{k-1} u) \\
=
- 2 \frac d{dt} \Re (u \partial_t^{k+1} (|u|^4) , i\partial_t^{k-2} u) 
+2 \Re (u \partial_t^{k+2} (|u|^4) , i\partial_t^{k-2} u)\\
\equiv 2 \Re (u \partial_t^{k+2} (|u|^4) , i\partial_t^{k-2} u).
\end{multline}
Notice that at the last step we have used 
\begin{equation*}
\frac d{dt} \Re (u \partial_t^{k+1} (|u|^4) , i\partial_t^{k-2} u)\in {\mathcal P}_{2k,k_0}
\end{equation*} 
for a suitable $k_0$, whose proof is similar to the one of \eqref{tarsi}.
 By iterating $k-2$ times the argument we used from \eqref{tzbe} to \eqref{tzbenew}
(namely to move one time derivatives from the right factor to the left factor),  
we deduce \eqref{tinavita}
in the general case $k\geq 3$.
\\
 \begin{rem}{\rm 
Notice that the expression on the r.h.s. in \eqref{tinavita} may no longer be handled in a simple way, by
developing time derivatives by Leibniz rule and by replacing time derivatives with space derivatives (using the equation). In fact, following this direction, we could get at least one term where $4k$ space derivatives
are shared on two factors only, and hence we do not fit with the functionals in
\eqref{family1}, \eqref{family2}, \eqref{family3}. 
 In fact a finer analysis is needed to 
deal with the term
$
\Re (u \partial_t^{2k-1} (|u|^4) , i\partial_t u).
$
The key point is that, as we shall see below,
the expression $\Re (u \partial_t^{2k-1} (|u|^4) , i\partial_t u)$ has the advantage to fit in the analysis involving the same quantities $N$, $J$ and $T$ performed in the case $k=1$
(see subsection \ref{k=1}).}
\end{rem}
Notice that 
\begin{equation}\label{partpal}
 (u \partial_t^{2k-1} (|u|^4) , |u|^4 u)\in {\mathcal L}_{2k, k_0}
 \end{equation}
 for a suitable $k_0$,
 whose proof follows the same argument to get \eqref{impoterm}, namely 
 first expand time derivatives through Leibniz rule and then use the equation solved by $u(t,x)$ to substitute time derivatives for space derivatives. We then  easily deduce that in this way we 
 get multilinear expressions that fit in \eqref{family3} (in fact in the expression \eqref{partpal}
 are involved $2k-1$ time derivatives, and hence at most $4k-2$ space derivatives).
\\
By combining \eqref{tinavita}, \eqref{partpal} and using the equation solved by $u(t,x)$ we get
\begin{multline}\label{aklim}
I\equiv 2 (- 1)^{k-2}   \Re (u \partial_t^{2k-1} (|u|^4) , \partial_x^2 u)
\\=  (- 1)^{k-2} (\partial_{x}^2 (|u|^2), \partial_t^{2k-1} (|u|^4)) -2 (- 1)^{k-2}(|\partial_{x} u|^2, \partial_t^{2k-1} (|u|^4)) 
.\end{multline}
We claim that we have
\begin{equation}\label{aklimstar}
I\equiv (- 1)^{k-1} N_0 +(- 1)^{k-1}J_0 +(- 1)^{k-1} N_1
\end{equation}
where:
\begin{align}\label{N0,N1,J0}
N_0:=\int \partial_t^{2k-1} \partial_x^2 N N^2, 
\quad 
N_1:=\int \partial_t^{2k-1} N (\partial_x N)^2,
\quad 
J_0:=\int \partial_t^{2k-1} N J^2,
\end{align}
($N$, $J$
are defined in \eqref{defN}, \eqref{defJ}).
\\ 
Notice that the first term
on the right hand-side of \eqref{aklim}  is equivalent to $N_0$, up to a multiplicative factor.
Indeed we have
\begin{multline*}
(\partial_{x}^2 (|u|^2), \partial_t^{2k-1} (|u|^4))=\int  \partial_t^{2k-1} (N^2)\partial_{x}^2 N
\\
=\frac d{dt} \int  \partial_t^{2k-2} (N^2)\partial_{x}^2 N-
\int  \partial_t^{2k-2} (N^2) \partial_t \partial_{x}^2  N
\end{multline*}
and by noticing
that $\frac d{dt} \int  \partial_t^{2k-2} (N^2)\partial_{x}^2 N \in {\mathcal P}_{2k, k_0}$ 
(see the proof of \eqref{tarsi}) we get
$$(\partial_{x}^2 (|u|^2), \partial_t^{2k-1} (|u|^4))\equiv -
\int  \partial_t^{2k-2} (N^2) \partial_t\partial_{x}^2  N.$$
We can repeat the argument above to move another time derivative on the second factor, namely 
$$(\partial_{x}^2 (|u|^2), \partial_t^{2k-1} (|u|^4))\equiv \int  \partial_t^{2k-3} (N^2)
\partial_t^2 \partial_{x}^2 N$$
and hence by iteration of this argument $2k-3$ times more we get
$$(\partial_{x}^2 (|u|^2), \partial_t^{2k-1} (|u|^4))\equiv - N_0,$$
where $N_0$ is defined in \eqref{N0,N1,J0}.
\\
The second term (modulo a multiplicative factor) on the r.h.s. of \eqref{aklim}
is given by 
\begin{align}
\label{qufipi}
(|\partial_{x} u|^2, \partial_t^{2k-1} (|u|^4)) &= (|\partial_{x} u|^2, \partial_t^{2k-1} (N^2))\\
& = 2 
(|\partial_{x} u|^2, N\partial_t^{2k-1} N) 
+\sum_{j=1}^{2k-2} c_j 
(|\partial_{x} u|^2, \partial_t^{2k-1-j} N \partial_t^{j}N).\nonumber
\end{align}
By following the same argument as in the proof of \eqref{impoterm}
it is easy to show that $(|\partial_{x} u|^2, \partial_t^{2k-1-j} N \partial_t^{j}N)\equiv 0$
(in fact by recalling $N=|u|^2$, by using the Leibniz rule to develop time derivatives and by replacing time derivatives with space derivatives due to the equation solved by $u(t,x)$, we get a linear combination of terms belonging to \eqref{family2}).
Going back to \eqref{qufipi} we get
\begin{multline*}
(|\partial_{x} u|^2, \partial_t^{2k-1} (|u|^4))\equiv
2 (|\partial_{x} u|^2, N \partial_t^{2k-1} N)\\
=\frac 12 \int  (J^2+ (\partial_x N)^2)
\partial_t^{2k-1} N=\frac 12 J_0 + \frac 12 N_1,
\end{multline*}
where we used the identity \eqref{eleele} (see \eqref{N0,N1,J0} for the definition of $J_0, N_1$). Hence the proof of \eqref{aklimstar} is complete.
 \\
 Going back to \eqref{guido}, \eqref{degi}, by recalling that $II\equiv 0$
 and due to \eqref{aklimstar}, we may conclude \eqref{homogen}, provided that we prove the following statement. 
\begin{lemme}\label{fin}
We have the following relations:
\begin{align}\label{mbmb}
J_0+\frac{3}{2}N_1\equiv  -\frac{1}{2} N_0,
\quad \quad 
&
J_0\equiv 0\,, \quad \quad
N_0\equiv N_1,
\quad \quad
\end{align}
In particular $ N_0\equiv N_1\equiv
J_0 \equiv 0.$
\end{lemme}
\begin{proof}
We focus on the first identity in \eqref{mbmb}:
\begin{align}\label{sinceev}
J_0+N_1 & = \int \partial_t^{2k-1} N J^2 + \int \partial_t^{2k-1} N (\partial_x N)^2
\\\nonumber
& =4 \int \partial_t^{2k-1} N N |\partial_x u|^2 \equiv \int  \partial_t^{2k-1}  N N T
+  \int \partial_t^{2k-1} N\, N\, \partial^2_x N,
\end{align}
where we have used the definition of $T$ (see \eqref{defT}), \eqref{eleele} and we have neglected the 
contribution to $T$ given by  $-\frac 43 N^3$.
More specifically we have used the fact that $\int  \partial_t^{2k-1}  N N^4\in {\mathcal L}_{2k, k_0}$ for a suitable $k_0$,
whose proof is similar to the proof of \eqref{partpal}. 
\\
Next 
by \eqref{contin}
we can write
\begin{eqnarray}\label{complale}
 \int  \partial_t^{2k-1} N\, N\, T & = & -\int  \partial_t^{2k-2} \partial_x J \, N \,  T
= 
\int  \partial_t^{2k-3} \partial_x^2 T \, N\, T
 \\\nonumber
 & = &
 -\int  \partial_t^{2k-3} \partial_x T \, \partial_x N\,  T -\int  \partial_t^{2k-3} \partial_x T\,  N\,  \partial_x T.
\end{eqnarray}
We first notice that 
\begin{equation}\label{guidalb}
\int  \partial_t^{2k-3} \partial_x T\, N \, \partial_x T\equiv 0.
\end{equation}
In fact we have
\begin{multline}\label{threeter}\int  \partial_t^{2k-3} \partial_x T\, N \, \partial_x T=\frac d{dt} 
\int  \partial_t^{2k-4} \partial_x T\, N \, \partial_x T \\ {}-
\int  \partial_t^{2k-4} \partial_x T\, \partial_t N \, \partial_x T-\int  \partial_t^{2k-4} \partial_x T\, N \, \partial_t \partial_x T.\end{multline}
We claim that 
\begin{equation}\label{imsantub}
\frac d{dt} 
\int  \partial_t^{2k-4} \partial_x T\, N \, \partial_x T\in {\mathcal P}_{2k,k_0}, \quad
\int  \partial_t^{2k-4} \partial_x T\, \partial_t N \, \partial_x T \in {\mathcal L }_{2k,k_0}
\end{equation} for a suitable $k_0$.
For simplicity we sketch the proof of the second fact in \eqref{imsantub} (the proof of the first one is similar
once we combine the argument below and the proof of \eqref{tarsi}).
If we replace $T$ by its expression (see \eqref{defT})
then we are reduced to show
$$\int  \partial_t^{2k-4} \partial_x^3 N\, \partial_t N \, \partial_x^3N, 
\int  \partial_t^{2k-4} \partial_x^3 N\, \partial_t N \, \partial_x (|\partial_x u|^2),$$$$
\int  \partial_t^{2k-4} \partial_x^3 N\, \partial_t N \, \partial_x (N^3),
\int  \partial_t^{2k-4} \partial_x (|\partial_x u|^2)\, \partial_t N \, \partial_x (|\partial_x u|^2),
$$$$\int  \partial_t^{2k-4} \partial_x (|\partial_x u|^2)\, \partial_t N \, \partial_x (N^3),
\int  \partial_t^{2k-4} \partial_x (N^3)\, \partial_t N \, \partial_x (N^3)\in {\mathcal L }_{2k,k_0}.$$
We shall prove for instance
\begin{equation}\label{anne}\int  \partial_t^{2k-4} \partial_x^3 N\, \partial_t N \, \partial_x^3N\in {\mathcal L }_{2k,k_0},\end{equation}
the other terms above can be treated in a similar way.
The proof of this fact is similar to the proof of \eqref{impoterm}.
In fact by expanding space and time derivatives, by using the equation and hence transforming each time derivatives in two space derivatives, one is reduced to a linear combination
of multilinear expressions of derivatives of $u, \bar u$ (the amount of space derivatives involved is $4k$). In the multilinear expressions we have either at least four factors with nontrivial derivatives,
or we have exactly three factors with nontrivial derivatives of order at least $2$. Then in any case we get terms belonging to 
\eqref{family1} and \eqref{family2}, and we conclude the proof of \eqref{anne}.

Summarizing from \eqref{threeter} and \eqref{imsantub} we get:
\begin{align}\label{threeternew}\int  \partial_t^{2k-3} \partial_x T\, N \, \partial_x T\equiv -\int  \partial_t^{2k-4} \partial_x T\, N \, \partial_t \partial_x T,\end{align}
and by iteration of this argument (that allows to move one time derivatives from the first factor
to the third factor) we get
\begin{multline}\label{dalla}
\int  \partial_t^{2k-3} \partial_x T\, N \, \partial_x T \equiv (-1)^{k-2} \int  \partial_t^{k-2} \partial_x T N \partial_t^{k-1} \partial_x T
\\ = \frac{(-1)^{k-2}}2 \int \partial_t ((\partial_t^{k-2}(\partial_x T))^2) N\\
 =\frac{(-1)^{k-2}}2 \frac d{dt} \int (\partial_t^{k-2}(\partial_x T))^2 N \\
{}- \frac{(-1)^{k-2}}2 \int 
 (\partial_t^{k-2}(\partial_x T))^2
 \partial_t N
\equiv 0,
 \end{multline}
and hence we get \eqref{guidalb}.
Notice that at the last step in \eqref{dalla} we used
$$\frac d{dt} \int (\partial_t^{k-2}(\partial_x T))^2 N  \in {\mathcal P}_{2k,k_0}
\quad \int (\partial_t^{k-2}(\partial_x T))^2 \partial_t N \in {\mathcal L}_{2k,k_0},
$$
for a suitable $k_0\in\N$, whose proof follows by the same argument
used to prove \eqref{tarsi} and \eqref{impoterm}, once we replace $T$ by its explicit expression
(see \eqref{defT}) as we did above to prove \eqref{imsantub}.\\ 
\\
Next we go back to \eqref{complale} and we notice that  we can expand the 
first term on the r.h.s. as follows:
\begin{align}\label{Guidalb}\int  \partial_t^{2k-3} \partial_x T \partial_x N T
&\equiv -\int  \partial_t^{2k-3} \partial_x T \partial_x N \partial_x^2 N
= -\frac 12 \int  \partial_t^{2k-3} \partial_x T \partial_x (\partial_x N)^2
\\\nonumber&= \frac 12 \int  \partial_t^{2k-3} \partial_x^2 T (\partial_x N)^2
= -\frac 12 \int  \partial_t^{2k-2} \partial_x J (\partial_x N)^2
\\\nonumber &=\frac 12  \int  \partial_t^{2k-1} N (\partial_x N)^2=\frac 12 N_1
\end{align}
where we have used \eqref{contin} and we have replaced at the first step $T$ by $-\partial_x^2 N$
(we have neglected the other contributions coming from the expression of $T$ since, arguing as above, we can be shown that they are equivalent to zero since they produce terms belonging either to \eqref{family2} or to \eqref{family3}).
Summarizing, by combining \eqref{sinceev}, \eqref{complale}, \eqref{guidalb} and \eqref{Guidalb} we get 
$$
J_0+N_1\equiv \int \partial_t^{2k-1} N\, N\, \partial_x^2 N-\frac 12 N_1.
$$
We claim that 
\begin{equation}\label{Kogen}\int \partial_t^{2k-1} N N \partial_x^2 N\equiv -\frac 12 N_0
\end{equation}
and it will conclude the proof of the first equivalence in \eqref{mbmb}.
Indeed we have
\begin{equation}\label{matha}\partial_t^{2k-1} (N^2)=2 N \, \partial_t^{2k-1} N
+\sum_{j=1}^{2k-2} c_j \partial_t^j N\, \partial_t^{2k-1-j} N
\end{equation}
where $c_j$ are suitable real numbers.
Notice that 
\begin{equation}\label{mathann}
\int  \partial_t^j N \partial^{2k-1-j}_t  N
\partial_x^2 N
\in {\mathcal L}_{2k, k_0}, \quad \forall j\in\{1,\dots , 2k-2\}.
\end{equation}
for a suitable $k_0$, whose proof is similar to  \eqref{impoterm}, we skip the details.
By combining \eqref{mathann} and \eqref{matha} we get
$$\int \partial_t^{2k-1} (N^2)\partial_x^2 N\equiv \int 2 N \, \partial_t^{2k-1} N \partial_x^2 N$$ and hence \eqref{Kogen} follows provided that
\begin{equation*}\int \partial_t^{2k-1} (N^2)\partial_x^2 N\equiv -N_0.\end{equation*}
In turn this last identity follows by iterating $(2k-1)$-times the
following computation:
\begin{align}\label{lastimpref}N_0&= \int \partial_t ( N^2 \partial_t^{2k-2} \partial_x^2 N)
-\int \partial_t (N^2) \partial_t^{2k-2} \partial_x^2 N
\\\nonumber
&=\frac d{dt} \int N^2 \partial_t^{2k-2} \partial_x^2 N
-\int \partial_t (N^2) \partial_t^{2k-2} \partial_x^2 N
\equiv -\int \partial_t (N^2) \partial_t^{2k-2} \partial_x^2 N,\end{align}
where we have used $\frac d{dt} \int N^2 \partial_t^{2k-2} \partial_x^2 N\in 
{\mathcal P}_{2k, k_0}$ for a suitable $k_0$ (the proof is similar to the proof of \eqref{tarsi}).
\\

Concerning the second identity in \eqref{mbmb}, first notice that by an iterated application of \eqref{contin}  one can deduce 
\begin{equation}\label{diffbp}
\int \partial_t^{2k-1} N\, J^2 \equiv (-1)^k
\int \partial_x^{4k-3} J\, J^2.\end{equation} 
Indeed, by using \eqref{contin} we get
\begin{equation}\label{kappel}
\int \partial_t^{2k-1} N J^2=-\int \partial_t^{2k-2} \partial_x J J^2
=\int \partial_t^{2k-3} \partial_x^2 T J^2\equiv -\int \partial_t^{2k-3} \partial_x^4 N J^2\,,
\end{equation}
where the last step follows once we replace $T$ by its expression (see \eqref{defT}) and we keep
only the contribution given by $-\partial_x^2 N$. 
In fact, we claim that the other contributions coming from the remaining two terms in $T$ 
are equivalent to zero, namely: 
\begin{equation}\label{model}\int \partial_t^{2k-3} \partial_x^2 (|\partial_x u|^2) J^2\in {\mathcal L}_{2k, k_0}, \quad 
\int \partial_t^{2k-3} \partial_x^2 (N^3) J^2\in {\mathcal L}_{2k, k_0}\end{equation}
for a suitable $k_0$,
whose proof is similar to the proof of \eqref{impoterm}.
Going back to \eqref{kappel}, by an iteration of the same argument
(that allows to change two time derivatives with a fourth order space derivative), we get
$$\int \partial_t^{2k-1} N J^2
\equiv (-1)^{k+1}
\int \partial_t \partial_x^{4k-4} N J^2$$
and hence we conclude
\eqref{diffbp} by using \eqref{contin}. By using \eqref{diffbp} and integration by parts we get
\begin{equation}\label{quasiJ}J_0=\int \partial_t^{2k-1} N\, J^2\equiv (-1)^k \int \partial^{4k-3}_x J \,J^2=
2(-1)^{k+1}  \int \partial_x^{4k-4} J\,  J\, \partial_x J.\end{equation}
Again by integration by parts we get
$$J_0=2(-1)^{k}  \int \partial_x^{4k-5} J\,  \partial_x J\, \partial_x J
+2(-1)^{k}  \int \partial_x^{4k-5} J\,  J\, \partial_x^2 J\equiv 2(-1)^{k}  \int \partial_x^{4k-5} J\,  J\, \partial_x^2 J,$$
where we have used the property $\int \partial_x^{4k-5} J\,  \partial_x J\, \partial_x J
\in {\mathcal L}_{2k, k_0}$ for a suitable $k_0$, whose proof
follows again by looking at the proof of \eqref{impoterm}.
Hence by iteration of the previous argument we get 
 \begin{align*} J_0\equiv 
 2(-1)^{k}  \int \partial_x J J\, \partial_x^{4k-4} J& \equiv 
 (-1)^{k+1}  \int J^2\, \partial_x^{4k-3} J\equiv -J_0,\end{align*} 
where at the last step we have used \eqref{diffbp}.
Summarizing we get $J_0\equiv -J_0$ which implies $J_0\equiv 0$.\\

Next, we prove the third identity in \eqref{mbmb}. By using Leibniz rule 
with respect to the time variable and integration by parts we get
\begin{align}\label{tizio}
N_0&=\frac d{dt} \int N^2\, \partial_t^{2k-2} \partial_x^2 N-\int \partial_t (N^2)\, \partial_t^{2k-2} \partial_x^2 N\\\nonumber &\equiv -\int \partial_t (N^2)\, \partial_t^{2k-2} \partial_x^2 N=-2\int \partial_t N\, N\,  \partial_t^{2k-2} \partial_x^2 N
\\\nonumber
&=2 \int \partial_t \partial_x N\, N\,  \partial_t^{2k-2} \partial_x N
+2\int \partial_t N\, \partial_x N \, \partial_t^{2k-2} \partial_x N
\end{align}
where we have used $\frac d{dt} \int N^2\, \partial_t^{2k-2} \partial_x^2 N\in 
{\mathcal P}_{2k, k_0}$ for a suitable $k_0$ (see the proof of \eqref{tarsi}).
We claim that 
\begin{equation}\label{aquilone}
\int \partial_t \partial_x N\, N\,  \partial_t^{2k-2} \partial_x N\equiv 0.
\end{equation}
In fact notice that 
\begin{multline}\label{partinico}\int \partial_t \partial_x N\, N\,  \partial_t^{2k-2} \partial_x N=
\frac d{dt} \int \partial_t \partial_x N\, N\,  \partial_t^{2k-3} \partial_x N
\\{}- \int \partial_t \partial_x N\, \partial_t N\,  \partial_t^{2k-3} \partial_x N
-\int \partial_t^2 \partial_x N\, N\,  \partial_t^{2k-3} \partial_x N
\end{multline}
and since 
$$\int \partial_t \partial_x N\, \partial_t N\,  \partial_t^{2k-3} \partial_x N\in {\mathcal L}_{2k,k_0},
\quad \frac d{dt} \int \partial_t \partial_x N\, N\,  \partial_t^{2k-3} \partial_x N\in 
{\mathcal P}_{2k, k_0}$$
(whose proof is similar to the proof of \eqref{impoterm} and \eqref{tarsi}) we get 
\begin{equation}\label{martino}\int \partial_t \partial_x N\, N\,  \partial_t^{2k-2} \partial_x N
\equiv -\int \partial_t^2 \partial_x N\, N\,  \partial_t^{2k-3} \partial_x N.\end{equation}
By iteration of the proof to get \eqref{martino}
(that allows to move one time derivative from the third factor to the first factor)
we get 
$$\int \partial_t \partial_x N\, N\,  \partial_t^{2k-2} \partial_x N
\equiv (-1)^{k-1}\int \partial_t^{k} \partial_x N\, N\,  \partial_t^{k-1} \partial_x N
$$
and we conclude \eqref{aquilone} since
$$\int \partial_t^{k} \partial_x N\, N\,  \partial_t^{k-1} \partial_x N=\frac 12 
\int \partial_t (\partial_t^{k-1} \partial_x N)^2 N$$$$=\frac 12 \frac {d}{dt}
\int (\partial_t^{k-1} \partial_x N)^2 N -\frac 12 \int (\partial_t^{k-1} \partial_x N)^2 
\partial_t N\equiv 0$$
where at the last step we have used
$$\frac {d}{dt}
\int (\partial_t^{k-1} \partial_x N)^2 N\in {\mathcal P}_{2k, k_0}, \quad
\int (\partial_t^{k-1} \partial_x N)^2 
\partial_t N\in {\mathcal L}_{2k, k_0}$$ 
for some $k_0$, whose proof is similar to the proof of
\eqref{tarsi} and
\eqref{impoterm}.
Consequently from \eqref{tizio} and \eqref{aquilone} we get
\begin{equation}\label{mm}N_0\equiv 
2\int \partial_t N\, \partial_x N \, \partial_t^{2k-2} \partial_x N.
\end{equation}
We conclude provided that 
\begin{equation}\label{mitlef}2\int \partial_t N\, \partial_x N \, \partial_t^{2k-2} \partial_x N\equiv N_1.
\end{equation}
In order to prove \eqref{mitlef},  notice 
\begin{equation}\label{foa}\partial_t^{2k-2} (\partial_x N )^2=2\partial_x N \, \partial_t^{2k-2} \partial_x N
+\sum_{j=1}^{2k-3} c_j \partial_t^j \partial_x N \partial_t^{2k-2-j} \partial_x N\end{equation}
where $c_j$ are suitable real numbers.
We have 
\begin{equation}\label{claimlast}
\int \partial_t N \partial_t^j \partial_x N \partial^{2k-2-j}_t \partial_x N
\in {\mathcal L}_{2k, k_0}, \quad \forall j\in\{1,\dots , 2k-3\}\end{equation}
for a suitable $k_0$, whose proof is similar to the proof of \eqref{impoterm}.
Hence necessarily by \eqref{foa} we get the desired conclusion
$$2\int \partial_t N\, \partial_x N \, \partial_t^{2k-2} \partial_x N\equiv \int \partial_t N\,\partial_t^{2k-2} (\partial_x N)^2\equiv N_1.$$
Notice that the last equivalence $\equiv$ is obtained by iteration of the following computation,
that allows to move one time derivative from the second factor to the first factor in the intermediate term of the previous equivalence:
\begin{align*}\int \partial_t N\,\partial_t^{2k-2} (\partial_x N)^2&= \frac d{dt} \int \partial_t N\,\partial_t^{2k-3} (\partial_x N)^2 - \int \partial_t^2 N\,\partial_t^{2k-3} (\partial_x N)^2
\\\nonumber &\equiv - \int \partial_t^2 N\,\partial_t^{2k-3} (\partial_x N)^2
\end{align*}
where we have used 
$\frac d{dt} \int \partial_t N\,\partial_t^{2k-3} (\partial_x N)^2\in {\mathcal P}_{2k, k_0}$
for some $k_0$, whose proof is similar to the proof of \eqref{tarsi}.
This completes the proof of Lemma~\ref{fin}.

\end{proof}

The proof of Theorem~\ref{determinist} in the case $k\geq 2$, $M=\infty$ is now complete.
\subsection{Proof of Theorem~\ref{determinist} for $M<\infty$}
Let us finally consider the case $M<\infty$.  First notice that if $u(t,x)$ solves \eqref{NLS_N} then
$u_M(t,x)=\pi_M u(t,x)$ solves
\begin{equation}\label{u_M}
\partial_t u_M=i\partial_x^2  u_M-i|u_M|^4u_M+i(1-\pi_M)\big( |u_M|^4 u_M\big)\,,
\end{equation}
namely $u_M(t,x)$ is an exact solution to NLS up to the extra term 
$i(1-\pi_M)\big( |u_M|^4 u_M\big)$.
It is worth mentioning that the energy $E_{2k}(u)$ associated with
the infinite dimensional equation NLS, has the following structure
\begin{equation}\label{energyP}E_{2k} (u)=\int \partial_x^{2k} u 
\partial_x^{2k} \bar u+ \sum_{l=0}^{l_k}{c_l} \int p_l(u)
\end{equation}
where $c_l\in \C$ and 
$p_l(u)\in {\mathcal P}_{2k}$ are suitable densities, where 
$$
{\mathcal P}_{2k}=
\Big\{\prod_{i=1}^N \partial_{x}^{\alpha_i} v_i \,|\, v_i\in \{u, \bar u\},\, \sum_{i=1}^N \alpha_i\leq  
4k-2,\, \alpha_i\leq 2k-1,\, N\geq 3\Big\}.
$$
It is worth mentioning that the functionals defined in ${\mathcal P}_{2k}$ belong to the family
\eqref{family3}, and hence in particular we have
${\mathcal P}_{2k}\subset {\mathcal L}_{2k, N}$.

Notice that if we have 
$$
p_l(u)=\prod_{i=1}^N \partial_{x}^{\alpha_i} v_i\in {\mathcal P}_{2k}
$$
then along any time dependent function $u=u(t,x)$ we compute
\begin{equation}\label{stome}
\frac d{dt} \int p_l(u)=\sum_{j=1}^N \int \big(\prod_{i=1,\dots,N, i\neq j} \partial_{x}^{\alpha_i} v_i
\big)\partial_{x}^{\alpha_j} \partial_t v_j.
\end{equation}
If moreover $u(t,x)=u_M(t,x)$ solves \eqref{u_M} then we are allowed to replace 
$\partial_t v_j$ in \eqref{stome} either by $i\partial_x^2  u_M-i|u_M|^4u_M
+i(1-\pi_M)\big( |u_M|^4 u_M\big)$ in the case $v_j=u_M$ 
or by $ -i\partial_x^2  \bar u_M+i|u_M|^4\bar u_M-i(1-\pi_M)\big( |u_M|^4 \bar u_M\big)$)
in the case  $v_j=\bar u_M$.  Motivated by this fact we introduce for every 
$$
p_l(u)=\prod_{i=1}^N \partial_{x}^{\alpha_i} v_i\in {\mathcal P}_{2k}
$$
the new functional 
$\tilde p_l(u)$ 
which is the expression obtained at the r.h.s. of \eqref{stome} when replacing 
$ \partial_t v_j$ by $i\partial_x^2  u-i|u|^4u$ in the case  $v_j=u$  or by
$-i\partial_x^2  \bar u+i|u|^4\bar u$  in case $v_j=\bar u$. Namely
$$
\tilde p_l(u)=\sum_{j=1}^N \int (\prod_{i=1,\dots ,N, i\neq j} \partial_{x}^{\alpha_i} v_i
)\partial_{x}^{\alpha_j} w, 
$$
where $w=i\partial_x^2  u-i|u|^4u$ when  $v_j=u$ and $w=-i\partial_x^2  \bar u+i|u|^4\bar u$ when $v_j=\bar u$.
\\

We also introduce $p_{l,M}^*(u)$  defined by
the expression obtained at the r.h.s. of \eqref{stome} when replacing 
$ \partial_t v_j$ by $i(1-\pi_M)\big( |u|^4 u\big) $ in the case $v_j=u$ and by  $-i(1-\pi_M)\big( |u|^4 \bar u\big)$ in the case $v_j=\bar u$.
Namely
$$
p^*_{l,M}(u)=\sum_{j=1}^N \int (\prod_{i=1,\dots ,N, i\neq j} \partial_{x}^{\alpha_i} v_i
)\partial_{x}^{\alpha_j}  w, 
$$
where $w=i(1-\pi_M)\big( |u|^4 u\big)$ in the case $v_j=u$ and $w=-i(1-\pi_M)\big( |u|^4 \bar u\big)$ in the case $v_j=\bar u$.
\\

As a consequence of the previous discussion, we deduce that
\begin{multline*}
\frac d{dt} E_{2k} (u_M)= \int \partial_x^{2k} \partial_t u_M 
\partial_x^{2k} \bar u_M + \int \partial_x^{2k}  u_M 
\partial_x^{2k} \partial_t \bar u_M\\
{}+\sum_{l=1}^{l_{k}}c_l \int \tilde p_l(u_M) + 
\sum_{l=1}^{l_{k}} c_l \int p_{l,M}^*(u_M).
\end{multline*}
We can replace   $\partial_t u_M $ by the r.h.s. in \eqref{u_M} and we obtain
\begin{multline}\label{cortona3}
\frac d{dt} E_{2k} (u_M)
= \int \partial_x^{2k} (i\partial_x^2  u_M-i|u_M|^4u_M)
\partial_x^{2k} \bar u_M 
\\{}+ \int \partial_x^{2k}  u_M 
\partial_x^{2k} (-i\partial_x^2  \bar u_M+i|u_M|^4 \bar u_M)
\\{}+\sum_{l=1}^{l_{k}}c_l \int \tilde p_l(u_M) 
+ \int \partial_x^{2k}  (i(1-\pi_M)\big( |u_M|^4 u_M\big))
\partial_x^{2k} \bar u_M 
\\{}+ \int \partial_x^{2k}  u_M 
\partial_x^{2k} (i\partial_x^2  (-i(1-\pi_M)\big( |u_M|^4 \bar u_M\big)) 
+ \sum_{l=1}^{l_{k}} c_l \int p_{l,M}^*(u_M).
\end{multline}
In the previous section, we proved that of $u$ is a smooth solution of \eqref{NLS} then it satisfies 
\begin{equation}\label{cortona}
\frac d{dt}E_{2k} (u)=F^{(\infty)}_{2k}(u),
\end{equation}
with $F^{(\infty)}_{2k}(u)$ satisfying the needed bounds. 
Using \eqref{stome}, \eqref{cortona} and taking the trace at $t=0$, we  obtain that every $C^\infty(\T)$ function $v_{0}$ (where the notation emphasizes that $v_{0}$ is not a solution to any equation)  satisfies the relation 
\begin{align}\label{cortona2}
\int \partial_x^{2k} (i\partial_x^2 v_{0}-i|v_{0}|^4v_{0}) \partial_x^{2k} \bar v_{0}
&+\int  \partial_x^{2k}v_{0} \partial_x^{2k}  (-i \partial_x^2 \bar{v}_{0}+i|v_{0}|^4\bar{v}_{0}) \\\nonumber &+\sum_{l=1}^{l_k} c_{l}\tilde{p}_l(v_{0})=F^{(\infty)}_{2k}(v_{0})\,.
\end{align}
Applying \eqref{cortona2} to $v_{0}=u_M$, we obtain that the sum of the first three terms at the r.h.s. of \eqref{cortona3}
is equal to $F_{2k}^\infty(u_M)$. On the other hand the fourth and the fifth terms are zero by orthogonality. More precisely,  $u_M$ is localized on the frequencies
$\{-M,\dots ,0, \dots, M\}$ and on the remaining factor we have the projection on the orthogonal Fourier modes.
\\

The proof will be completed, provided we show that
$$|p_{l,M}^*(u_M)|\lesssim \big(1+\|u\|_{H^{2k-1}}^{m_0}\big)$$ 
for some $m_0$. For that purpose, it is sufficient to estimate expressions of the following type:
$$\int \prod_{i=1}^N \partial^{\alpha_i}_x v_i dx, \quad v_i\in \{u_M, \bar u_M, 
(1-\pi_M)\big( |u_M|^4 u_M\big), (1-\pi_M)\big( |u_M|^4 \bar u_M\big)\}$$
where $\sum_{i=1}^N \alpha_i\leq 4k-2, \alpha_i\leq 2k-1$. 
Combining  the H\"older inequality and the Sobolev embedding $H^1\subset L^\infty$
it is easy to show that 
$$\big|
\int \prod_{i=1}^N \partial^{\alpha_i}_x v_i dx
\big|
\lesssim \prod_{i=1}^N \|v_i\|_{H^{2k-1}}
$$
and of course we conclude the desired estimate
since 
$$
\|(1-\pi_M)\big( |u_M|^4 \bar u_M\big)\|_{H^{2k-1}} \leq \||u_M|^4 \bar u_M\|_{H^{2k-1}}\lesssim \|u_M\|_{H^{2k-1}}^4\, ,
$$
where we have used at the last step the fact that $H^{2k-1}$ is an algebra for $k\geq 1$.
\section{Proof of Theorem~\ref{th3} }
As in the proof of Theorem~\ref{th2}, we consider the following truncated version of \eqref{NLS_focus}
\begin{equation}\label{NLS_N_focus}
(i\partial_t+\partial_x^2)u=-\pi_M\big(|\pi_M u|^4 \pi_M u\big),\,\, u(0,x)=u_0(x), \quad t\in \R,\,\, x\in\T. 
\end{equation}
Thanks to the $L^2$ conservation law, for $M<\infty$, we can define the global flow of \eqref{NLS_N_focus} and denote it by $\Phi_M(t)$.
However for $u_0\in B_R$ (see \eqref{biar}), we have bounds on the solutions, uniform in $M$ only on the interval $[-T,T]$, where $T>0$ is depending on $R$.
For $M=\infty$, we can define the solution of \eqref{NLS_N_focus} for data in $B_R$ only locally in the time interval $[-T,T]$, $T=T(R)$.
One can observe that the proof of Theorem~\ref{determinist} yields the following statement in the context of \eqref{NLS_N_focus}.
\begin{theoreme}
Let $k\geq 1$ be an integer. 
There is an  integer $m_0=m_0(k)>0$ and a positive constant $C$ such that the following holds true. There exists a functional $E_{2k}(u)$ such that
$$
E_{2k}(u)=\|u\|_{H^{2k}}^2 + R_{2k}(u),\quad E_{2k}(0)=0,
$$ 
where $R_{2k}$ satisfies \eqref{NEW}  and moreover for every  $M\in \N\cup \{\infty\}$, the solution of  \eqref{NLS_N_focus}  with data $u_0\in B_R$ satisfies 
\begin{equation*}
\frac d{dt} E_{2k}(\pi_M u(t))= F^{(M)}_{2k}(\pi_M u(t)),\quad t\in [-T,T],\, T=T(R)
\end{equation*}
where the functional $F^{(M)}_{2k}(u)$ satisfies \eqref{NEW_pak_2k} and \eqref{NEW_pak_2}.
\end{theoreme}
Now, the proof of Theorem~\ref{th3} can be done exactly as the proof of Theorem~\ref{th2} once we replace Proposition~\ref{cauchy-sigma} and 
Proposition~\ref{sigma} with the following local in time analogues in the context of \eqref{NLS_N_focus}. 
\begin{proposition}\label{cauchy-sigma_focus}
For every $R>0$ there is $C>0$ depending on $R$, such that for every $M\in \N\cup \{\infty\}$, if the initial data in \eqref{NLS_N_focus} satisfies
$
\|u_0\|_{H^{2k-\frac 12 -\epsilon}}<R,
$
then the corresponding solution satisfies
$$
\|u(t)\|_{H^{2k-\frac 12 -\epsilon}}<C,\quad t\in [-T,T],\, T=T(R).
$$
\end{proposition}
\begin{proposition}\label{sigma_focus}
For every $R>0$ there is $T>0$ such that the following holds true.
For every $A\subset H^{2k-\frac 12 -\epsilon}$ a compact set included in $B_R$, for every $\delta>0$ 
there exists $M_0\in \N$ such that for every $M\geq M_0$, every $u_0\in A$, the  local in time solution 
$$
(i\partial_t+\partial_x^2)u=-|u|^4 u,\,\, u(0,x)=u_0(x), \quad t\in [-T,T],\,\, x\in\T
$$
satisfies 
$
u(t)\in  \Phi_M(t) (A+B_\delta),
$
$
\forall t\in [-T, T]\, .
$
\end{proposition}

\end{document}